\documentclass[11pt,a4paper]{article}
\usepackage{amsmath,amssymb,amsthm,fullpage,graphicx}

\newtheorem{thm}{Theorem}[section]
\newtheorem{lemma}[thm]{Lemma}
\newtheorem{propn}[thm]{Proposition}
\newtheorem{cor}[thm]{Corollary}

\newtheorem{defn}[thm]{Definition}
\newtheorem{rem}[thm]{Remark}
\newtheorem{example}[thm]{Example}
\newtheorem{assu}[thm]{Assumption}

\newcommand{\bn}{\mathbb N}
\newcommand{\br}{\mathbb R}
\newcommand{\bz}{\mathbb Z}
\newcommand{\cq}{\mbox{$\mathcal Q$}}
\newcommand{\cx}{\mbox{$\Theta$}}
\newcommand{\Lte}{{\mathbb L}^2 ({\hat V},{\hat \mu}_0)}
\newcommand{\hf}{{\hat \Phi}}
\newcommand{\hh}{{\hat H}}

\allowdisplaybreaks

\let\OLDthebibliography\thebibliography
\renewcommand\thebibliography[1]{
  \OLDthebibliography{#1}
  \setlength{\parskip}{0pt}
  \setlength{\itemsep}{0pt plus 0.3ex}
}

\begin{document}

\title{Time-changes of stochastic processes\\associated with resistance forms}
\author{D.~A.~Croydon, B.~M.~Hambly and {T.~Kumagai}}
\date{July 25, 2016}
\maketitle

\begin{abstract}
\noindent
Given a sequence of resistance forms that converges with respect to the Gromov-Hausdorff-vague topology and satisfies a uniform volume doubling condition, we show the convergence of corresponding Brownian motions and local times. As a corollary of this, we obtain the convergence of time-changed processes. Examples of our main results include scaling limits of Liouville Brownian motion, the Bouchaud trap model and the random conductance model on trees and self-similar fractals. For the latter two models, we show that under some assumptions the limiting process is a FIN diffusion on the relevant space.
\end{abstract}



\setcounter{tocdepth}{2}

\section{Introduction}

In recent years, interest in time-changes of stochastic processes according to irregular measures has arisen from various sources. Fundamental examples of such time-changed processes include the so-called \emph{Fontes-Isopi-Newman (FIN) diffusion} \cite{FIN}, the introduction of which was motivated by the study of the localisation and aging properties of physical spin systems, and the two-dimensional \emph{Liouville Brownian motion} \cite{Beres,GRV}, which is the diffusion naturally associated with planar Liouville quantum gravity. More precisely, the FIN diffusion is the time-change of one-dimensional Brownian motion by the positive continuous additive functional with Revuz measure given by
\begin{equation}\label{onedfinmeasure}
\nu(dx)=\sum_{i}v_i\delta_{x_i}(dx),
\end{equation}
where $(v_i ,x_i)_{i\in\mathbb{N}}$ is the Poisson point process with intensity $\alpha v^{-1-\alpha}dvdx$, and $\delta_{x_i}$ is the probability measure placing all its mass at $x_i$. Similarly, the two-dimensional Liouville Brownian motion is the time-change of two-dimensional Brownian motion by the positive continuous additive functional with Revuz measure given by
\begin{equation}\label{2dlmeasure}
\nu(dx)=e^{\kappa\gamma(x)-\frac{\kappa^2}{2}\mathbf{E}(\gamma(x)^2)}dx
\end{equation}
for some $\kappa\in(0,2)$, where $\gamma$ is the massive Gaussian free field; actually the latter description is only formal since the Gaussian free field can not be defined as a function in two dimensions. In both cases, connections have been made with discrete models; the FIN diffusion is known to be the scaling limit of the one-dimensional Bouchaud trap model \cite{BC, FIN} and the constant speed random walk amongst heavy-tailed random conductances in one-dimension \cite{CernyEJP}, and the two-dimensional Liouville Brownian motion is conjectured to be the scaling limit of simple random walks on random planar maps \cite{GRV}, see also \cite{DS}. The goal here is to provide a general framework for studying such processes and their discrete approximations in the case when the underlying stochastic process is strongly recurrent, in the sense that it can be described by a resistance form, as introduced by Kigami (see \cite{Kig} for background). In particular, this includes the case of Brownian motion on tree-like spaces and low-dimensional self-similar fractals.

To present our main results, let us start by introducing the types of object under consideration (for further details, see Section \ref{prelimsec}). Let $\mathbb{F}$ be the collection of quadruples of the form $(F,R,\mu,\rho)$, where: $F$ is a non-empty set; $R$ is a resistance metric on $F$ such that $(F,R)$ is complete, separable and locally compact, and moreover closed balls in $(F,R)$ are compact; $\mu$ is a locally finite Borel regular measure of full support on $(F,R)$; and $\rho$ is a marked point in $F$. Note that the resistance metric is associated with a resistance form $(\mathcal{E},\mathcal{F})$ (see Definition \ref{resformdef} below), and we will further assume that for elements of $\mathbb{F}$ this form is regular in the sense of Definition \ref{regulardef}. In particular, this ensures the existence of a related regular Dirichlet form $(\mathcal{E},\mathcal{D})$ on $L^2(F,\mu)$, which we suppose is recurrent, and also a Hunt process $((X_t)_{t\geq 0},\:P_x,\: x\in F)$ that can be checked to admit jointly measurable local times $(L_t(x))_{x\in F,t\geq 0}$. The process $X$ represents our underlying stochastic process (i.e.\ it plays the role that Brownian motion does in the construction of the FIN diffusion and Liouville Brownian motion), and the existence of local times means that when it comes to defining the time-change additive functional, it will be possible to do this explicitly.

Towards establishing a scaling limit for discrete processes, we will assume that we have a sequence $(F_n,R_n,\mu_n,\rho_n)_{n\geq 1}$ in $\mathbb{F}$ that converges with respect to the Gromov-Hausdorff-vague topology (see Section \ref{ghpsec}) to an element $(F,R,\mu,\rho)\in\mathbb{F}$. Our initial aim is to show that it is then the case that the associated Hunt processes $X^n$ and their local times $L^n$ converge to $X$ and $L$, respectively. To do this we assume some regularity for the measures in the sequence -- this requirement is formalised in Assumption \ref{a1}, which depends on the following volume growth property. In the statement of the latter, we denote by $B_n(x,r)$ the open ball in $(F_n,R_n)$ centred at $x$ and of radius $r$, and also $r_0(n):= \inf_{x,y\in F_n,\:x\neq y}R_n(x,y)$ and $r_\infty(n):=\sup_{x,y\in F_n}R_n(x,y)$. We note that this control on the volume yields an equicontinuity property for the local times.

{\defn A sequence $(F_n,R_n,\mu_n,\rho_n)_{n\geq 1}$ in $\mathbb{F}$ is said to satisfy \emph{uniform volume growth with volume doubling (UVD)} if there exist constants $c_1,c_2,c_3\in(0,\infty)$ such that
\[c_1v(r)\leq \mu_n\left(B_{n}(x,r)\right)\leq c_2v(r),\qquad\forall x\in F_n,\:r\in[r_0(n),r_\infty(n)+1]\]
for every $n\geq 1$, where $v:(0,\infty)\rightarrow(0,\infty)$ is non-decreasing function with $v(2r)\leq c_3v(r)$ for every $r\in\mathbb{R}_+$.}

{\assu\label{a1} The sequence $(F_n,R_n,\mu_n,\rho_n)_{n\geq 1}$ in $\mathbb{F}$ satisfies UVD, and also
\begin{equation}\label{ghpconv}
\left(F_n,R_n,\mu_n,\rho_n\right)\rightarrow \left(F,R,\mu,\rho\right),
\end{equation}
in the Gromov-Hausdorff-vague topology, where $(F,R,\mu,\rho)\in\mathbb{F}$.}
\bigskip

It is now possible to state our first main result. We write $D(\mathbb{R}_+,M)$ for the space of cadlag processes on $M$, equipped with the usual Skorohod $J_1$ topology. The definition of equicontinuity of the local times $L^n$, $n\geq 1$, should be interpreted as the conclusion of Lemma~\ref{ltcont}.

{\thm\label{main1} Suppose Assumption \ref{a1} holds. It is then possible to isometrically embed $(F_n,R_n)$, $n\geq 1$, and $(F,R)$ into a common metric space $(M,d_M)$ in such a way that if $X^n$ is started from $\rho_n$, $X$ is started from $\rho$, then
\[\left(X^n_t\right)_{t\geq 0}\rightarrow \left(X_t\right)_{t\geq 0}\]
in distribution in $D(\mathbb{R}_+,M)$. Moreover, the local times of $L^n$ are equicontinuous, and if the finite collections $(x_i^n)_{i=1}^k$ in $F_n$, $n\geq 1$, are such that $d_M(x_i^n,x_i)\rightarrow 0$ for some $(x_i)_{i=1}^k$ in $F$, then it simultaneously holds that
\begin{equation}\label{ltconv1}
\left(L^n_t\left(x_i^n\right)\right)_{i=1,\dots, k, t\geq 0}\rightarrow \left(L_t\left(x_i\right)\right)_{i=1,\dots, k, t\geq 0},
\end{equation}
in distribution in $C(\mathbb{R}_+,\mathbb{R}^k)$.}
\bigskip

From the above result, we further deduce the convergence of time-changed processes. The following assumption adds the time-change measure to the framework.

{\assu\label{a2} Assumption \ref{a1} holds with (\ref{ghpconv}) replaced by
\[\left(F_n,R_n,\mu_n,\nu_n,\rho_n\right)\rightarrow \left(F,R,\mu,\nu,\rho\right),\]
in the (extended) Gromov-Hausdorff-vague topology (see Section \ref{ghpsec}), where $\nu_n$ is a locally finite Borel regular measure on $F_n$, and $\nu$ is a locally finite Borel regular measure on $(F,R)$ with $\nu(F)>0$.}
\bigskip

The time-change additive functional that we consider is the following:
\begin{equation}\label{atdef}
A_t:=\int_FL_t(x)\nu(dx).
\end{equation}
In particular, let $\tau(t):=\inf\{s>0:\:A_s>t\}$ be the right-continuous inverse of $A$, and define a process $X^\nu$ by setting
\begin{equation}
\label{xnudef}
X^\nu_{t}:=X_{\tau(t)}.
\end{equation}
As described in Section \ref{rfsec}, this is the trace of $X$ on the support of $\nu$ (with respect to the measure $\nu$), and its Dirichlet form is given by the corresponding Dirichlet form trace. We define $A^n$, $\tau^n$, and $X^{n,\nu_n}$ similarly. The space $L^1_{\rm loc}(\mathbb{R}_+,M)$ is the space of cadlag functions $\mathbb{R}_+\rightarrow M$ such that $\int_0^Td_M(\rho,f(t))dt<\infty$ for all $T\geq0$, equipped with the topology induced by supposing $f_n\rightarrow f$ if and only if $\int_0^Td_M(f_n(t),f(t))dt\rightarrow 0$ for any $T\geq 0$.

{\cor\label{maincor} (a) Suppose Assumption \ref{a2} holds, and that $\nu$ has full support. Then it is possible to isometrically embed $(F_n,R_n)$, $n\geq 1$, and $(F,R)$ into a common metric space $(M,d_M)$ in such a way that
\begin{equation}\label{xnnu}
X^{n,\nu_n}\rightarrow X^\nu
\end{equation}
in distribution in $D(\mathbb{R}_+,M)$, where we assume that $X^n$ is started from $\rho_n$, and $X$ is started from $\rho$.\\
(b) Suppose Assumption \ref{a2} holds, and that $X$ is continuous. Then \eqref{xnnu} holds in distribution in $L^1_{\rm loc}(\mathbb{R}_+,M)$.}
\bigskip

The above results are proved in Section~\ref{copsec}, following the introduction of preliminary material in Section~\ref{prelimsec}. In the remainder of the article, we demonstrate the application of Theorem~\ref{main1} and Corollary~\ref{maincor} to a number of natural examples. Firstly, we investigate the Liouville Brownian motion associated with a resistance form, showing in Proposition~\ref{lbmconv} that Assumption~\ref{a1} implies the convergence of the corresponding Liouville Brownian motions. This allows us to deduce the convergence of Liouville Brownian motions on a variety of trees and fractals, which we discuss in Example~\ref{lbmexamples}. We note that Liouville Brownian motion associated with a resistance form is a toy
model and we discuss it merely as a simple example of our methods. The more interesting and challenging problem of analysing this process in two dimensions is not possible within our framework. Next, in Section~\ref{bouchsec}, we proceed similarly for the Bouchaud trap model, describing the limiting process as the FIN process associated with a resistance form in Proposition~\ref{btmresult}, and giving an application in Example~\ref{SG55ex}. Related to this, in Section~\ref{rcmsec}, we study the heavy-tailed random conductance model on trees and a class of self-similar fractals, discussing a FIN limit for the so-called constant speed random walk in Propositions~\ref{rcmtreeresult},~\ref{ssfrcmresult} and
Examples~\ref{exa6-5},~\ref{exa6-18}. Heat kernel estimates for the limiting FIN processes will be presented in
a forthcoming paper~\cite{CroHamKum}.

Of the applications outlined in the previous paragraph, one that is particularly illustrative of the contribution of this article is the random conductance model on the (pre-)Sierpi\'nski gasket graphs. More precisely, the random conductance model on a locally finite, connected graph $G=(V,E)$ is obtained by first randomly selecting edge-indexed conductances $(\omega_e)_{e\in E}$, and then, conditional on these, defining a continuous time Markov chain that jumps along edges with probabilities proportional to the conductances. For the latter process, there are two time scales commonly considered in the literature: firstly, for the \emph{variable speed random walk (VSRW)}, the jump rate along edge $e$ is given by $\omega_e$, so that the holding time at a vertex $x$ has mean $(\sum_{e:\:x\in e}\omega_e)^{-1}$; secondly, for the \emph{constant speed random walk (CSRW)}, holding times are assumed to have unit mean. From this description, it is clear that the CSRW is a time-change of the VSRW according to the measure placing mass $\sum_{e:\:x\in e}\omega_e$ on vertex $x$. Here, we will only ever consider conductances that are uniformly bounded below, but this still gives a rich enough model for there to exist a difference in the trapping behaviour experienced by the VSRW and CSRW. Indeed, in the one-dimensional case (i.e.\ when $G$ is $\mathbb{Z}$ equipped with edges between nearest neighbours) when conductances are i.i.d., it is easily checked that the VSRW has as its scaling limit Brownian motion (by adapting the argument of \cite[Appendix A]{CernyEJP} to the VSRW, for example); although the VSRW will cross edges of large conductance many times before escaping, it does so quickly, so that homogenisation still occurs. In the case of random conductances also uniformly bounded from above, the analogous result was proved in \cite{kk} for the VSRW on the fractal graphs shown in Figure \ref{sg}, with limit being Brownian motion on the Sierpi\'nski gasket. In Section \ref{rcmfractalsec}, we extend this result significantly to show the same is true whenever the conductance distribution has at most polynomial decay at infinity. Specifically, writing $X^{n,\omega}$ for the VSRW on the $n$th level graph and $X$ for Brownian motion on the Sierpi\'nski gasket, we prove that, under the annealed law (averaging over both process and environment),
\begin{equation}\label{sgvsrw}
\left(X^{n,\omega}_{5^nt}\right)_{t\geq 0}\rightarrow \left(X_t\right)_{t\geq0};
\end{equation}
the time scaling here is the same as for the VSRW on the unweighted graph. For the CSRW, on the other hand, the many crossings of edges of large conductance lead to more significant trapping, which remains in the limit. In particular, if the conductance distribution satisfies $\mathbf{P}(\omega_e>u)\sim u^{-\alpha}$ for some $\alpha\in(0,1)$, then, as noted above, in the one-dimensional case the CSRW has a FIN diffusion limit \cite{CernyEJP}. Applying our time-change results, we are able to show that the corresponding result holds for the Sierpi\'nski gasket graphs. Namely, writing $X^{n,\omega,\nu}$ for the CSRW on the $n$th level graph, we establish that there exists a constant $c$ such that, again under the annealed law,
\begin{equation}\label{sgcsrw}
\left(X^{n,\omega,\nu}_{c3^{n/\alpha}(5/3)^nt}\right)_{t\geq 0}\rightarrow\left(X^{\nu}_t\right)_{t\geq 0},
\end{equation}
where the limit is now $\alpha$-FIN diffusion on the Sierpi\'nski gasket, which  is time-change of the Brownian motion on the limiting gasket by a Poisson random measure defined similarly to (\ref{onedfinmeasure}), but with Lebesgue measure in the intensity replaced by the appropriate Hausdorff measure. (Note that, in the case that $\mathbf{E}\omega_e<\infty$, our techniques also yield convergence of CSRW to the Brownian motion, see Remark \ref{finitemoments}.) Full details for the preceding discussion are provided in Section \ref{rcmsec}. At the start of the latter section, we also give an expanded heuristic explanation for the appearance of the FIN diffusion as a limit of the CSRW amongst heavy-tailed conductances. We remark that the specific conclusion of this interpretation is dependent on the point recurrence of the processes involved; by contrast, for the random conductance model on $\mathbb{Z}^d$ for $d\geq 2$, the same trapping behaviour gives rise in the limit to the so-called \emph{fractional kinetics process}, for which the time-change and spatial motion are uncorrelated \cite{BarCern, CernyEJP}.

\begin{figure}[t]
\begin{center}
\scalebox{0.12}{\includegraphics{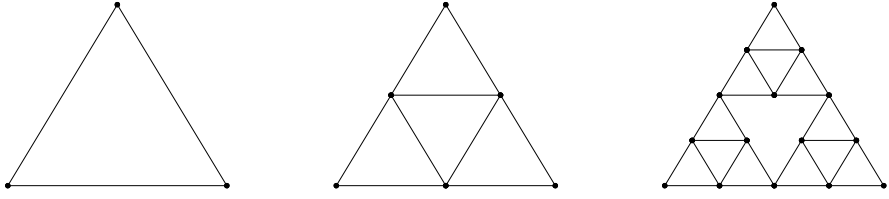}}
\end{center}
\vspace{-10pt}
\caption{The Sierpi\'nski gasket graphs $G_1$, $G_2$, $G_3$.}\label{sg}
\end{figure}

Finally, we note there are many other applications to which the notion of time-change is relevant, so that the techniques of this article might be useful. Although we do not consider it here, one such example is the \emph{diffusion on branching Brownian motion}, as recently constructed in \cite{AndHar}. Moreover, whilst the examples of time-changes described above are based on measures that are constant in time, our main results will also be convenient for describing time-changes based on space-time measures, i.e.\ via additive functionals of the form $A_t:=\int_{F\times\mathbb{R}_+}\mathbf{1}_{\{s\leq L_t(x)\}}\nu(dxds)$. In particular, Theorem~\ref{main1} would be well-suited to extending the study of the scaling limits of randomly trapped random walks, as introduced in \cite{BCCR}, from the one-dimensional setting to trees and fractals.

\section{Preliminaries}\label{prelimsec}

\subsection{Resistance forms and associated processes}\label{rfsec}

In this section, we define precisely the objects of study and outline some of their relevant properties; primarily this involves a recap of results from \cite{FOT} and \cite{Kig}. We start by recalling the definition of a resistance form and its associated resistance metric.

{\defn [{\cite[Definition 3.1]{Kig}}]\label{resformdef} Let $F$ be a non-empty set. A pair $(\mathcal{E},\mathcal{F})$ is called a \emph{resistance form} on $F$ if it satisfies the following five conditions.
\begin{description}
  \item[RF1] $\mathcal{F}$ is a linear subspace of the collection of functions $\{f:F\rightarrow\mathbb{R}\}$ containing constants, and $\mathcal{E}$ is a non-negative symmetric quadratic form on $\mathcal{F}$ such that $\mathcal{E}(f,f)=0$ if and only if $f$ is constant on $F$.
  \item[RF2] Let $\sim$ be the equivalence relation on $\mathcal{F}$ defined by saying $f\sim g$ if and only if $f-g$ is constant on $F$. Then $(\mathcal{F}/\sim,\mathcal{E})$ is a Hilbert space.
  \item[RF3] If $x\neq y$, then there exists a $f\in \mathcal{F}$ such that $f(x)\neq f(y)$.
  \item[RF4] For any $x,y\in F$,
  \begin{equation}\label{resdef}
  R(x,y):=\sup\left\{\frac{\left|f(x)-f(y)\right|^2}{\mathcal{E}(f,f)}:\:f\in\mathcal{F},\:\mathcal{E}(f,f)>0\right\}<\infty.
  \end{equation}
  \item[RF5] If $\bar{f}:=(f \wedge 1)\vee 0$, then $\bar{f}\in\mathcal{F}$ and $\mathcal{E}(\bar{f},\bar{f})\leq\mathcal{E}({f},{f})$ for any $f\in\mathcal{F}$.
\end{description}}
\bigskip

We note that \eqref{resdef} can be rewritten as
\[
R(x,y)=\left(\inf\left\{\mathcal{E}(f,f):\:
f\in \mathcal{F},\: f(x)=1,\: f(y)=0\right\}\right)^{-1},
\]
which is the effective resistance between $x$ and
$y$.
The function $R:F\times F\rightarrow \mathbb{R}$
is actually a metric on $F$ (see \cite[Proposition 3.3]{Kig}); we call this the \emph{resistance metric} associated with $(\mathcal{E},\mathcal{F})$. Henceforth, we will assume that we have a non-empty set $F$ equipped with a resistance form $(\mathcal{E},\mathcal{F})$ such that $(F,R)$ is complete, separable and locally compact. Defining the open ball centred at $x$ and of radius $r$ with respect to the resistance metric by $B_R(x,r):=\left\{y\in F:\:R(x,y)<r\right\}$, and denoting its closure by $\bar{B}_R(x,r)$, we will also assume that $\bar{B}_R(x,r)$ is compact for any $x\in F$ and $r>0$. Furthermore, we will restrict our attention to resistance forms that are regular, as per the following definition.

{\defn [{\cite[Definition 6.2]{Kig}}]\label{regulardef} Let $C_0(F)$ be the collection of compactly supported, continuous (with respect to $R$) functions on $F$, and $\|\cdot\|_F$ be the supremum norm for functions on $F$. A resistance form $(\mathcal{E},\mathcal{F})$ on $F$ is called \emph{regular} if and only if $\mathcal{F}\cap C_0(F)$ is dense in $C_0(F)$ with respect to $\|\cdot\|_F$.}
\bigskip

We next introduce related Dirichlet forms and stochastic processes. First, suppose $\mu$ is a Borel regular measure on $(F,R)$ such that $0<\mu(B_R(x,r))<\infty$ for all $x\in F$ and $r>0$. Moreover, write $\mathcal{D}$ to be the closure of $\mathcal{F}\cap C_0(F)$ with respect to the inner product $\mathcal{E}_1$ on $\mathcal{F}\cap L^2(F,\mu)$ given by
\begin{equation}\label{e1def}
\mathcal{E}_1(f,g):=\mathcal{E}(f,g)+\int_Ffgd\mu.
\end{equation}
Under the assumption that $(\mathcal{E},\mathcal{F})$ is regular, we then have the following. See \cite{FOT} for the definition of a regular Dirichlet form.

{\thm [{\cite[Theorem 9.4]{Kig}}] The quadratic form $(\mathcal{E},\mathcal{D})$ is a regular Dirichlet form on $L^2(F,\mu)$.}
\bigskip

Given a regular Dirichlet form, standard theory then gives us the existence of an associated Hunt process $((X_t)_{t\geq 0},\:P_x,\: x\in F)$ (e.g. \cite[Theorem 7.2.1]{FOT}). Note that such a process is, in general, only specified uniquely for starting points outside a set of zero capacity. However, in this setting every point has strictly positive capacity (see \cite[Theorem 9.9]{Kig}), and so the process is defined uniquely everywhere. Moreover, since we are assuming closed balls are compact, we have from \cite[Theorem 10.4]{Kig} that $X$ admits a jointly continuous transition density $(p_t(x,y))_{x,y\in F,t>0}$.
We note that the Dirichlet form for Brownian motion on $\br^d$ is a resistance form only when $d=1$. However, resistance forms are a rich class that contains various Dirichlet forms  for diffusions on fractals, see \cite{kig1}.

Key to this study will be the existence of local times for $X$. As a first step to introducing these, note that the strict positivity of the capacity of points remarked upon above implies that all points are regular (see \cite[Theorems 1.3.14 and 3.1.10, and Lemma A.2.18]{FukuChen}, for example). Thus $X$ admits local times everywhere (see \cite[(V.3.13)]{BG}). In the following lemma, by studying the potential density of $X$, we check that these local times can be defined in a jointly measurable way and satisfy an occupation density formula.

{\lemma (a) Define the (one-)potential density $(u(x,y))_{x,y\in F}$ of $X$ by setting
\begin{equation}\label{udef}
u(x,y)=\int_0^\infty e^{-t} p_t(x,y)dt.
\end{equation}
It then holds that $u(x,y)<\infty$ for all $x,y\in F$. Furthermore,
\begin{equation}\label{hitlap}
E_x\left(e^{-\tau_y}\right)=\frac{u(x,y)}{u(y,y)},
\end{equation}
where $\tau_y:=\inf\{t>0:\:X_t=y\}$ is the hitting time of $y$ by $X$,
and also
\begin{equation}\label{ufluc}
|u(x,y)-u(x,z)|^2\leq u(x,x)R(y,z)
\end{equation}
for all $x,y,z\in F$.\\
(b) The process $X$ admits jointly measurable local times $(L_t(x))_{x\in F,t\geq 0}$ that satisfy, $P_x$-a.s.\ for any $x$,
\begin{equation}\label{occdens}
\int_0^t \mathbf{1}_A(X_s)ds = \int_AL_t(y)\mu(dy)
\end{equation}
for all measurable subsets $A\subseteq F$ and $t\geq 0$.}
\begin{proof} To prove part (a), we essentially follow the proof of \cite[Theorem 7.20]{Barlow}, and then apply results from \cite{MR}. First, observe that the definition of the resistance metric at (\ref{resdef}) readily implies
\begin{equation}\label{rfluc}
\left|f(x)-f(y)\right|^2\leq \mathcal{E}(f,f)R(x,y)
\end{equation}
for all $f\in\mathcal{F}$, $x,y\in F$. Hence
\[f(x)^2\leq 2f(y)^2+2\left|f(x)-f(y)\right|^2\leq 2f(y)^2+2\mathcal{E}(f,f)R(x,y).\]
Using the local compactness of $(F,R)$, for any point $x\in F$, we can integrate the above over a compact neighbourhood of $x$ to obtain $f(x)^2\leq c\mathcal{E}_1(f,f)$ for any $f\in\mathcal{D}$, where $\mathcal{E}_1$ was defined at \eqref{e1def}. We thus have that $f\mapsto f(x)$ is a bounded linear operator on the Hilbert space $(\mathcal{D},\mathcal{E}_1^{1/2})$, and so by the Riesz representation theorem there exists a function $u(x,\cdot)\in\mathcal{D}$ such that
\begin{equation}\label{repro}
\mathcal{E}_1(u(x,\cdot),f)=f(x)
\end{equation}
for all $f\in\mathcal{D}$. From \eqref{repro}, we immediately obtain that $u(x,x)=\mathcal{E}_1(u(x,\cdot),u(x,\cdot))<\infty$. In combination with \eqref{rfluc}, this implies \eqref{ufluc} and the finiteness of $u(x,y)$ everywhere. Furthermore, if we define an operator on $L^2(F,\mu)$ by setting $Uf(x):=\int_Fu(x,y)f(y)\mu(dy)$, then by arguing exactly as in the proof of \cite[Theorem 7.20]{Barlow}, one can check $\mathcal{E}_1(Uf,g)=\int_Ffg d\mu$ for every $f\in C_0(F)$ and $g\in \mathcal{D}$. It follows that $U$ agrees with the resolvent of $X$ on $C_0(F)$, i.e. $Uf(x):=E_x\int_0^\infty e^{-t}f(X_t)dt$ for all $f\in C_0(F)$, and extending the latter statement to all $f\in L^2(F,\mu)$ is elementary. By the continuity of the transition density in this setting, this implies that the function $u$ can alternatively be defined via \eqref{udef}. To complete the proof of part (a), we note that \eqref{hitlap} is proved in \cite[Theorem 3.6.5]{MR}.

From part (a), we know that $E_x(e^{-\tau_y})$ is a jointly continuous function of $x,y\in F$. Thus, because we also know that all points of $F$ are regular for $X$, we can immediately apply the first part of \cite[Theorem 1]{GK} to obtain that $X$ admits jointly measurable local times $(L_t(x))_{x\in F,t\geq 0}$. Furthermore, since $X$ has a transition density, it holds that $\mu$ is a reference measure for $X$, i.e.\ $\mu(A)=0$ if and only if $U\mathbf{1}_A(x)=\int_0^\infty e^{-t}P_x(X_t\in A)dt=0$ for all $x\in F$ (see \cite[Definition V.1.1]{BG}). Thus we can apply the second part of \cite[Theorem 1]{GK} to confirm (\ref{occdens}) holds.
\end{proof}

We now describe background on time-changes of the Hunt process $X$ from \cite[Section 6.2]{FOT}. First suppose $\nu$ is an arbitrary positive Radon measure on $(F,R)$. As at (\ref{atdef}), define a continuous additive functional $(A_t)_{t\geq 0}$ by setting $A_t:=\int_FL_t(x)\nu(dx)$, and let $(\tau(t))_{t\geq 0}$ be its right-continuous inverse, i.e.\ $\tau(t):=\inf\left\{s>0:\:A_s>t\right\}$. If $G\subseteq F$ is the closed support of $\nu$, then $((\tilde{X})_{t\geq 0},\:P_x,\: x\in G)$ is also a strong Markov process, where $\tilde{X}_t:=X_{\tau(t)}$; this is the \emph{trace} of $X$ on $G$ (with respect to $\nu$). We also define a trace of the Dirichlet form $(\mathcal{E},\mathcal{D})$ on $G$, which we will denote by $(\tilde{\mathcal{E}},\tilde{\mathcal{D}})$, by setting
\begin{equation}\label{formtrace}
\tilde{\mathcal{E}}(g,g):=\inf\left\{\mathcal{E}(f,f):\:f\in\mathcal{D}_e,\:f|_{G}=g\right\},
\end{equation}
\begin{equation}\label{domaintrace}
\tilde{\mathcal{D}}:=\left\{g\in L^2(G,\nu):\:\tilde{\mathcal{E}}(g,g)<\infty\right\},
\end{equation}
where $\mathcal{D}_e$ is the extended Dirichlet space associated with $(\mathcal{E},\mathcal{D})$, i.e.\ the family of $\mu$-measurable functions $f$ on $F$ such that $|f|<\infty$, $\mu$-a.e.\, and there exists an $\mathcal{E}$-Cauchy sequence $(f_n)_{n\geq 0}$ in $\mathcal{D}$ such that $f_n(x)\rightarrow f(x)$, $\mu$-a.e. Connecting these two notions is the following result.

{\thm[{\cite[Theorem 6.2.1]{FOT}}]\label{trace} It holds that $(\tilde{\mathcal{E}},\tilde{\mathcal{D}})$ is a regular Dirichlet form on $L^2(G,\nu)$, and the associated Hunt process is $\tilde{X}$.}
\bigskip

Finally, we note a result that, in the recurrent case, characterises the trace of our Dirichlet form on a compact set. Note that the Dirichlet form $(\mathcal{E},\mathcal{D})$ is said to be recurrent if and only if $1\in \mathcal{D}_e$ and $\mathcal{E}(1,1)=0$.

{\lemma\label{goodtrace} If $(\mathcal{E},\mathcal{D})$ is recurrent and $G$ is compact, then $(\tilde{\mathcal{E}},\tilde{\mathcal{D}})$ is a regular resistance form on $G$, with associated resistance metric $R|_{G\times G}$.}
\begin{proof} Since $(\mathcal{E},\mathcal{D})$ is recurrent, we have that $\mathcal{D}_e=\mathcal{F}$ (see \cite[Proposition 2.13]{KL}).
Thus
\begin{equation}\label{etrace}
\tilde{\mathcal{E}}(g,g)=\inf\left\{\mathcal{E}(f,f):\:f\in\mathcal{F},\:f|_{G}=g\right\},
\end{equation}
and also $\tilde{\mathcal{D}}=\{f|_G:\:f\in\mathcal{F}\}\cap L^2(G,\nu)$. By (\ref{rfluc}), we moreover have that $\{f|_G:\:f\in\mathcal{F}\}\subseteq C(G)\subseteq L^2(G,\nu)$, and so
\begin{equation}\label{dtrace}
\tilde{\mathcal{D}}=\left\{f|_G:\:f\in\mathcal{F}\right\}.
\end{equation}
Finally, we observe that \eqref{etrace} and \eqref{dtrace} give that $(\tilde{\mathcal{E}},\tilde{\mathcal{D}})$ is the trace of the resistance form $(\mathcal{E},\mathcal{F})$ on $G$ in the sense of \cite[Definition 8.3]{Kig}. Since $G$ is closed, by \cite[Theorem 8.4]{Kig}, this implies $(\tilde{\mathcal{E}},\tilde{\mathcal{D}})$ is also a regular resistance form on this set, with associated resistance metric $R|_{G\times G}$.
\end{proof}

\subsection{Gromov-Hausdorff-vague topology}\label{ghpsec}

In this section we introduce the Gromov-Hausdorff-vague topology and an extension that we require. For more details regarding such metrics, see \cite{ADH, ALWtop}. We start by defining a topology on $\mathbb{F}_c$, which is the subset of $\mathbb{F}$ containing elements $(F,R,\mu,\rho)$ such that $(F,R)$ is compact. In particular, for two elements $(F,R,\mu,\rho),(F',R',\mu',\rho')\in\mathbb{F}_c$, we set $\Delta_c((F,R,\mu,\rho),(F',R',\mu',\rho'))$ to be equal to
\begin{eqnarray}\label{ghpmetric}
\inf_{M,\psi,\psi'}\left\{d_M^H\left(\psi(F),\psi'(F)\right)+d_M^P\left(\mu\circ\psi^{-1},\mu'\circ\psi'^{-1}\right)+d_M(\rho,\rho')\right\},
\end{eqnarray}
where the infimum is taken over all metric spaces $M=(M,d_M)$ and isometric embeddings $\psi:(F,R)\rightarrow (M,d_M)$, $\psi':(F',R')\rightarrow (M,d_M)$, and we define $d_M^H$ to be the Hausdorff distance between compact subsets of $M$, and $d_M^P$ to be the Prohorov distance between finite Borel measures on $M$. It is known that $\Delta_c$ defines a metric on the equivalence classes of $\mathbb{F}_c$ (where we say two elements of $\mathbb{F}_c$ are equivalent if there is a measure and root preserving isometry between them), see \cite[Theorem 2.5]{ADH}.

To extend $\Delta_c$ to a metric on the equivalence classes of $\mathbb{F}$, we consider bounded restrictions of elements of $\mathbb{F}$. More precisely, for $(F,R,\mu,\rho)\in \mathbb{F}$, define $(F^{(r)},R^{(r)},\mu^{(r)},\rho^{(r)})$ by setting: $F^{(r)}$ to be the closed ball in $(F,R)$ of radius $r$ centred at $\rho$, i.e.\ $\bar{B}_R(\rho,r)$; $R^{(r)}$ and $\mu^{(r)}$ to be the restriction of $R$ and $\mu$ respectively to  $F^{(r)}$, and $\rho^{(r)}$ to be equal to $\rho$. By assumption, $(F^{(r)},R^{(r)})$ is compact, and so to check that $(F^{(r)},R^{(r)},\mu^{(r)},\rho^{(r)})\in\mathbb{F}_c$ it will suffice to note that: $R^{(r)}$ is a resistance metric on $F^{(r)}$, the associated resistance form $(\mathcal{E}^{(r)},\mathcal{F}^{(r)})$ is regular, and $(\mathcal{E}^{(r)},\mathcal{F}^{(r)})$ is moreover a recurrent regular Dirichlet form. (These claims follow from Theorem \ref{trace} and Lemma \ref{goodtrace}.)

As in \cite[Lemma 2.8]{ADH}, we can check the regularity of the restriction operation with respect to the metric $\Delta_c$ to show that, for any two elements of the space $\mathbb{F}$, the map $r\mapsto \Delta_c((F^{(r)},R^{(r)},\mu^{(r)},\rho^{(r)}),(F'^{(r)},R'^{(r)},\mu'^{(r)},\rho'^{(r)}))$ is cadlag. (NB. In \cite[Lemma 2.8]{ADH}, the metric spaces are assumed to be length spaces, but it is not difficult to remove this assumption.) This allows us to define a function $\Delta$ on $\mathbb{F}^2$ by setting
\begin{eqnarray}
\lefteqn{\Delta\left((F,R,\mu,\rho),(F',R',\mu',\rho')\right)}\nonumber\\
&:=&\int_0^\infty e^{-r}\left(1\wedge\Delta_c((F^{(r)},R^{(r)},\mu^{(r)},\rho^{(r)}),(F'^{(r)},R'^{(r)},\mu'^{(r)},\rho'^{(r)}))
\right) dr,\label{ghvdef}
\end{eqnarray}
and one can check that this is a metric on (the equivalence classes of) $\mathbb{F}$, cf.\ \cite[Theorem 2.9]{ADH}, and also \cite[Proof of Proposition 5.12]{ALWtop}. The associated topology is the Gromov-Hausdorff-vague topology, as defined at \cite[Definition 5.8]{ALWtop}. From \cite[Proposition 5.9]{ALWtop}, we have the following important consequence of convergence in this topology.

{\lemma\label{embeddings} Suppose $(F_n,R_n,\mu_n,\rho_n)$, $n\geq 1$, and $(F,R,\mu,\rho)$ are elements of $\mathbb{F}$ such that $(F_n,R_n,\mu_n,\rho_n)\rightarrow(F,R,\mu,\rho)$ in the Gromov-Hausdorff-vague topology. It is then possible to embed $(F_n,R_n)$, $n\geq 1$, and $(F,R)$ isometrically into the same (complete, separable, locally compact) metric space $(M,d_M)$ in such a way that, for Lebesgue-almost-every $r\geq 0$,
\begin{equation}\label{embedconv}
d_M^H\left(F_n^{(r)},F^{(r)}\right)\rightarrow 0,\qquad d_M^P\left(\mu_n^{(r)},\mu^{(r)}\right)\rightarrow 0,\qquad d_M(\rho_n^{(r)},\rho^{(r)})\rightarrow 0,
\end{equation}
where we have identified the various objects with their embeddings.}
\bigskip

We next note that the measure bounds of UVD transfer to limits under the Gromov-Hausdorff-vague topology. The proof, which is an elementary consequence of the previous result, is omitted.

{\lemma Suppose $(F,R,\mu,\rho)\in\mathbb{F}$ is the limit with respect to the Gromov-Hausdorff-vague topology of a sequence $(F_n,R_n,\mu_n,\rho_n)_{n\geq 1}$ in $\mathbb{F}$ that satisfies UVD. It is then the case that
\begin{equation}\label{uvdforlimit}
c_1v(r)\leq \mu\left(B_{R}(x,r)\right)\leq c_2v(r),\qquad\forall x\in F,\:r\in[r_0,r_\infty+1],
\end{equation}
where $r_0:= \inf_{x,y\in F,\:x\neq y}R(x,y)$ and $r_\infty:=\sup_{x,y\in F}R(x,y)$.}
\bigskip

Finally, we define an extended version of the Gromov-Hausdorff-vague topology for elements of the form $(F,R,\mu,\nu,\rho)$, where $(F,R,\mu,\rho)\in\mathbb{F}$, and $\nu$ is another locally finite Borel regular measure on $(F,R)$ (not necessarily of full support). We do this in the obvious way: for elements  $(F,R,\mu,\nu,\rho)$ and $(F',R',\mu',\nu',\rho')$ such that $(F,R)$ and $(F',R')$ are compact, we include the term $d_M^P\left(\nu\circ\psi^{-1},\nu'\circ\psi'^{-1}\right)$ in the definition of $\Delta_c$ at (\ref{ghpmetric}); in the general case, we use this version of $\Delta_c$ to define $\Delta((F,R,\mu,\nu,\rho),(F',R',\mu',\nu',\rho'))$ as at (\ref{ghvdef}); the induced topology is then the extended Gromov-Hausdorff-vague topology. It is straightforward to check that the natural adaptation of Lemma \ref{embeddings} that includes the convergence $d_M^P(\nu_n^{(r)},\nu^{(r)})$ also holds, where $\nu_n^{(r)}$, $\nu^{(r)}$ is the restriction of $\nu_n$, $\nu$ to $F_n^{(r)}$, $F^{(r)}$, respectively.

\subsection{Local time continuity}

Key to our arguments is the following equicontinuity result for the local times of a sequence satisfying the UVD property. Since the proof is similar to the discrete time version proved for graphs in \cite[Theorem 1.2]{CroyLT}, we only provide a sketch.

{\lemma\label{ltcont} If $(F_n,R_n,\mu_n,\rho_n)_{n\geq 1}$ is a sequence in $\mathbb{F}_c$ satisfying $\sup_nr_\infty(n)<\infty$ and also UVD, then, for each $\varepsilon>0$ and $T>0$,
\[\lim_{\delta\rightarrow0}\sup_{n\geq 1}\sup_{x\in F_n}{P}^{n}_x\left(\sup_{\substack{y,z\in F_n:\\R_n(y,z)\leq\delta}}\sup_{0\leq t\leq T}{\left|L^{n}_{t}(y)-L_t^{n}(z)\right|}\geq \varepsilon\right)=0.\]}

\begin{proof} We start by checking the commute time identity for a resistance form. In particular, if $(F,R,\mu,\rho)\in\mathbb{F}_c$, then we claim that
\begin{equation}\label{commute}
E_x\left(\tau_y\right)+E_y\left(\tau_x\right)=R(x,y)\mu(F)\qquad\forall x,y\in F,
\end{equation}
where $\tau_z$ is the hitting time of $z$ by $X$. Indeed, fix $x,y\in F$. As in the proof of \cite[Proposition 4.2]{Kum}, there exists a function $g_{\{x\}}(y,\cdot)\in \mathcal{F}$ such that: $\mathcal{E}(g_{\{x\}}(y,\cdot),f)=f(y)$ for every $f\in\mathcal{F}$ such that $f(x)=0$; $g_{\{x\}}(y,y)=\mathcal{E}(g_{\{x\}}(y,\cdot),g_{\{x\}}(y,\cdot))=R(x,y)$; and also $g_{\{x\}}(y,x)=0$. By symmetry, we deduce that
\[\mathcal{E}\left(g_{\{x\}}(y,\cdot)+g_{\{y\}}(x,\cdot),f\right)=\mathcal{E}\left(g_{\{x\}}(y,\cdot),f-f(x)\right)+\mathcal{E}\left(g_{\{y\}}(x,\cdot),f-f(y)\right)=0\]
for every $f\in \mathcal{F}$. It follows that $g_{\{x\}}(y,\cdot)+g_{\{y\}}(x,\cdot)$ is constant, and so satisfies
\[g_{\{x\}}(y,\cdot)+g_{\{y\}}(x,\cdot)\equiv g_{\{x\}}(y,x)+g_{\{y\}}(x,x)=R(x,y).\]
Moreover, as at \cite[(4.7)]{Kum}, we have that $g_{\{x\}}(y,\cdot)$ is the occupation density for $X$, started at $y$ and killed at $x$, and so $E_y(\tau_x)=\int_Fg_{\{x\}}(y,z)\mu(dz)$. Combining the latter two results, the identity at (\ref{commute}) follows.

We now suppose $(F_n,R_n,\mu_n,\rho_n)_{n\geq 1}$ is a sequence in $\mathbb{F}_c$ as in the statement of the lemma, and consider the associated local time processes. From \cite[(V.3.28)]{BG}, we have that
\begin{equation}\label{tailforlt}
{P}^{n}_x\left(\sup_{0\leq t\leq T}{\left|L^{n}_{t}(y)-L_t^{n}(z)\right|}\geq \varepsilon\right)\leq 2e^Te^{-\varepsilon/2\delta_n(x,y)},
\end{equation}
where
\[\delta_n(x,y)^2:=1-E_x^n\left(e^{-\tau^n_y}\right)E_y^n\left(e^{-\tau^n_y}\right)\leq E_x^n\left(\tau^n_y\right)+E_y^n\left(\tau^n_x\right)=R_n(x,y)\mu_n(F_n),\]
and the final equality is a consequence of (\ref{commute}). Hence we obtain that
\[\sup_{x,y,z\in F_n}{P}^{n}_x\left(\sup_{0\leq t\leq T}\frac{\left|L^{n}_{t}(y)-L_t^{n}(z)\right|}{\sqrt{R_n(y,z)\mu_n(F_n)}}\geq \varepsilon\right)\leq 2e^Te^{-\varepsilon/2}.\]
Thus if we set
\[\Gamma_n:=\int_{F_n}\int_{F_n}
\exp\Big(\frac{\sup_{0\leq t\leq T}\left|L^{n}_{t}(y)-L_t^{n}(z)\right|}{4\sqrt{R_n(y,z)\mu_n(F_n)}}\Big)\mu_n(dy)\mu_n(dz),\]
then it follows that
\begin{eqnarray}\label{Gamest}
\lim_{\lambda\rightarrow\infty}\sup_{n\geq 1}\sup_{x\in F_n}P_x^n\left(\Gamma_n> \lambda \mu_n(F_n)^2\right)=0.
\end{eqnarray}
The result now follows from a standard argument involving Garsia's lemma, as originally proved in \cite{Garsia}, see also \cite{GRR}; applications to local times appear in \cite{BP, CroyLT}, for example. We simply highlight the differences. Choose $y,z\in F_n$ and $t\in[0,T]$. Then let $(K_i)_{i=0}^\infty$ be a sequence of balls $K_i=B_n(y,2^{1-2i}R_n(y,z))$, so that $K_0$ contains both $y$ and $z$, and $\cap_{i\geq 0}K_i=\{y\}$. Write $f_{K_i}:=\mu_n(K_i)^{-1}\int_{K_i}L^n_t(w)\mu_n(dw)$, and then we deduce that
\begin{eqnarray*}
\lefteqn{e^{|f_{K_i}-f_{K_{i-1}}|/16\sqrt{2^{-2i}R_n(y,z)\mu_n(F_n)}}}\\
&\leq& \frac{1}{\mu_n(K_i)\mu_n(K_{i-1})}\int_{K_i} \int_{K_{i-1}} e^{|L^n_t(w)-L^n_t(w')|/4\sqrt{R_n(w,w')\mu_n(F_n)}}\mu_n(dw)\mu_n(dw')\\
&\leq & cv(2^{1-2i}R_n(y,z))^{-2}\Gamma_n,
\end{eqnarray*}
where the first inequality is an application of Jensen's inequality, and the second is obtained from UVD and the definition of $\Gamma_n$. Summing over $i$ and repeating for a sequence decreasing to $z$ yields
\begin{eqnarray}\label{lteee}
\left|L^{n}_{t}(y)-L_t^{n}(z)\right|\leq 16\sqrt{R_n(y,z)\mu_n(F_n)}\sum_{i=0}^\infty2^{-i}\log\left(\frac{c\Gamma_n}{v(2^{1-2i}R_n(y,z))^{2}}\right).
\end{eqnarray}
Now, suppose $\Gamma_n\leq \lambda \mu_n(F_n)^2$. The UVD property then gives $\Gamma_n\leq c\lambda v(r_\infty(n))$. Together with the doubling property of $v$ and the assumption that $M=\sup_nr_\infty(n)<\infty$ we thus find that
\begin{eqnarray}
{\left|L^{n}_{t}(y)-L_t^{n}(z)\right|}
&\leq & 16\sqrt{R_n(y,z)v(r_\infty(n))}\sum_{i=0}^\infty2^{-i}\log\left(\frac{c\lambda v(r_\infty(n))^2}{v(2^{1-2i}R_n(y,z))^{2}}\right)\nonumber\\
&\leq & c\sqrt{R_n(y,z)v(M)}\max\{1,\log \lambda^{1/c}M, \log R_n(y,z)^{-1}\},\nonumber
\end{eqnarray}
uniformly over $y,z\in F_n$ and $t\in[0,T]$. Combining this estimate with (\ref{Gamest}) completes the proof.
\end{proof}

Note that we also have continuity of the limiting local times.

{\lemma\label{ltcontlimit} If $(F,R,\mu,\rho)\in\mathbb{F}_c$ satisfies (\ref{uvdforlimit}), then the local times $(L_t(x))_{x\in F,t\geq 0}$ of the associated process are continuous in $x$, uniformly over compact intervals of $t$, $P_y$-a.s.\ for any $y\in F$.}
\begin{proof} Arguing as for (\ref{Gamest}), we have that
\[\Gamma:=\int_{F}\int_{F}e^{\sup_{0\leq t\leq T}\left|L_{t}(y)-L_t(z)\right|/4\sqrt{R(y,z)\mu(F)}}\mu(dy)\mu(dz)\]
is a finite random variable, $P_y$-a.s., for any $T<\infty$. Hence, by applying the estimate (\ref{lteee}), we obtain the result.
\end{proof}

\section{Convergence of processes}\label{copsec}

\subsection{Compact case}\label{compsec}

In this section, we prove the first part of Theorem \ref{main1} in the case that the metric spaces $(F_n,R_n)$, $n\geq 1$, and $(F,R)$ are all compact (see Proposition \ref{compactcase} below). Throughout, we assume that Assumption \ref{a1} holds. Note that, by Lemma \ref{embeddings}, under this Gromov-Hausdorff-vague convergence assumption, it is possible to suppose that $(F_n,R_n)$, $n\geq 1$, and $(F,R)$ are isometrically embedded into a common metric space $(M,d_M)$ such that
\begin{equation}\label{hconv}
d_M^H\left(F_n,F\right)\rightarrow 0,\qquad d_M^P\left(\mu_n,\mu\right)\rightarrow 0,\qquad d_M(\rho_n,\rho)\rightarrow 0,
\end{equation}
where we have identified the various objects with their embeddings. Throughout this section, we fix one such collection of embeddings.

Our argument will depend on approximating the processes $X^n$, $n\geq 1$, and $X$ by processes on finite state spaces. We start by describing such a procedure in the limiting case. Let $(x_i)_{i\geq 1}$ be a dense sequence of points in $F$ with $x_1=\rho$. For each $k$, it is possible to choose $\varepsilon_k$ such that
\begin{equation}\label{kcov}
F\subseteq\cup_{i=1}^kB_M(x_i,\varepsilon_k),
\end{equation}
(where $B_M(x,r)$ represents a ball in $(M,d_M)$,) and moreover one can do this in such a way that $\varepsilon_k\rightarrow 0$ as $k\rightarrow\infty$. Choose $\varepsilon^k_1,\varepsilon^k_2,\dots,\varepsilon^k_k\in [\varepsilon_k,2\varepsilon_k]$ such that $(B_M(x_i,\varepsilon_i^k))_{i=1}^k$ are continuity sets for $\mu$ (i.e.\ $\mu(\bar{B}_M(x_i,\varepsilon_i^k)\backslash B_M(x_i,\varepsilon_i^k))=0$); such a choice is possible because, for any $x\in M$, the map $r\mapsto \mu(B_M(x,r))$ has a countable number of discontinuities.
Define sets $K^k_1,K^k_2,\dots,K^k_k$ by setting $K^k_1=\bar{B}_M(x_1,\varepsilon_1^k)$ and
\begin{equation}
K^k_{i+1}=\bar{B}_M(x_{i+1},\varepsilon^k_{i+1})\backslash \cup_{j=1}^i \bar{B}_M(x_{j},\varepsilon^k_j).\label{eq:setdec}
\end{equation}
In particular, the elements of the collection $(K^k_i)_{i=1}^k$ are measurable, disjoint continuity sets, and cover $F$. We introduce a corresponding measurable mapping $\phi^{(k)}:F\rightarrow \{x_1,\dots,x_k\}$ by setting $\phi^{(k)}(x)=x_i$ if $x\in K^k_i$, and a related measure $\mu^{(k)}=\mu\circ (\phi^{(k)})^{-1}$. Of course, the image of $\phi^{(k)}$ might not be the whole of $\{x_1,\dots,x_k\}$ since some of the $K_i^k$ might be empty. So, to better describe it, we introduce the notation $I_k:=\{i:\:K_i^k\neq\emptyset\}$ and $V_k:=\{x_i:\:i\in I_k\}$. (We will often implicitly use the fact that the points $(x_i)_{i\in I_k}$ are distinct, which follows from the definition.) The following simple lemma establishes that the measure $\mu^{(k)}$ charges all the points of $V_k$.

{\lemma\label{muksupport} The support of the measure $\mu^{(k)}$ is equal to $V_k$.}
\begin{proof} Suppose $i\in \{1,\dots,k\}$ and $\mu^{(k)}(\{x_i\})=0$. Then by definition
\[0=\mu(K_i^k)=\mu\left(\bar{B}_R(x_{i},\varepsilon^k_{i})\backslash \cup_{j=1}^{i-1} \bar{B}_R(x_{j},\varepsilon^k_j)\right)=\mu\left({B}_R(x_{i},\varepsilon^k_{i})\backslash \cup_{j=1}^{i-1} \bar{B}_R(x_{j},\varepsilon^k_j)\right),\]
where we use that ${B}_R(x_{i},\varepsilon^k_{i})$ is a continuity set for $\mu$. Now,  ${B}_R(x_{i},\varepsilon^k_{i})\backslash \cup_{j=1}^{i-1} \bar{B}_R(x_{j},\varepsilon^k_j)$ is an open set. Thus, because $\mu$ has full support, the fact that the latter set has zero measure implies that it is empty. Hence ${B}_R(x_{i},\varepsilon^k_{i})\subseteq \cup_{j=1}^{i-1} \bar{B}_R(x_{j},\varepsilon^k_j)$. Since the right-hand side is closed, it follows that $\bar{B}_R(x_{i},\varepsilon^k_{i})\subseteq \cup_{j=1}^{i-1} \bar{B}_R(x_{j},\varepsilon^k_j)$, and therefore $K_i^k=\emptyset$. Thus $i\not\in I_k$. In particular, we have established that the support of $\mu^{(k)}$ contains $V_k$. Since the reverse inclusion is trivial, this completes the proof.
\end{proof}

Next observe that $\sup_{x\in F}R(x,\phi^{(k)}(x))\leq 2\varepsilon_k\rightarrow 0$, and hence $\mu^{(k)}\rightarrow \mu$ weakly as measures on $F$. This will allow us to check that a family of associated time-changed processes $X^{(k)}$ converge to $X$. Indeed, set \[A^{(k)}_t=\int_FL_t(x)\mu^{(k)}(dx).\]
The continuity of the local times $L$ (see Lemma \ref{ltcontlimit}) then implies that, $P_\rho$-a.s., for each $t$,
\[A^{(k)}_t\rightarrow \int_FL_t(x)\mu(dx)=t.\]
Since the processes are increasing, this convergence actually holds uniformly on compact intervals (cf.\ the proof of Dini's theorem). Setting $\tau^{(k)}(t):=\inf\{s>0:\:A^{(k)}_t>s\}$, it follows that, $P_\rho$-a.s., $\tau^{(k)}(t)\rightarrow t$
uniformly on compact intervals. Composing with the process $X$ to define $X^{(k)}_t:=X_{\tau^{(k)}(t)}$, we thus obtain that $X^{(k)}_t\rightarrow X_t$ for all $t\geq 0$ such that $X$ is continuous at $t$, $P_\rho$-a.s. In particular, denoting by $T_X$ the set of times $t$ such that $P_\rho(X\mbox{ is continuous at }t)=1$, this implies the following finite dimensional convergence result.

{\lemma\label{l1} If $t_1,\dots,t_m\in T_X$, then $d_M(X^{(k)}_{t_i},X_{t_i})\rightarrow0$ for each $i=1,\dots,m$, as $k\to\infty$,  $P_\rho$-a.s.}
\bigskip

We next adapt the approximation argument to the processes $X^n$, $n\geq1$. By (\ref{hconv}), it is possible to choose $x_i^n\in F_n$ such that $d_M(x_i^n,x_i)\rightarrow 0$, with the particular choice $x_1^n=\rho_n$. Moreover, by (\ref{kcov}), it is possible to suppose that for each $k$ there exists an integer $n_k$ such that, for $n\geq n_k$, $F_n\subseteq \cup_{i=1}^kB_M(x_i,\varepsilon_k)$. Thus, for each $k$ and $n\geq n_k$ we can define a map $\phi^{n,k}:F_n\rightarrow \{x_1^n,\dots,x_k^n\}$ by setting $\phi^{n,k}(x)=x_i^n$ if $x\in K_i^k$. Note that
\begin{equation}\label{closepoints}
\lim_{k\rightarrow\infty}\limsup_{n\rightarrow\infty}\sup_{x\in F_n}R_n(x,\phi^{n,k}(x))\leq \lim_{k\rightarrow\infty}\limsup_{n\rightarrow\infty}\left(2\varepsilon_k+\sup_{i=1,\dots,k}d_M(x_i^n,x_i)\right)=0.
\end{equation}
We define $\mu_n^{(k)}=\mu_n\circ(\phi^{n,k})^{-1}$, and set
\[A^{n,k}_t=\int_{F_n}L^n_t(x)\mu_n^{(k)}(dx).\]
Moreover, let $\tau^{n,k}(t)=\inf\{s>0:\:A^{n,k}_t>s\}$, and define $X^{n,k}_t:=X^{n}_{\tau^{n,k}(t)}$. It is then straightforward to deduce the following lemma.

{\lemma\label{l2} The law of $X^{n,k}$ under $P^n_{\rho_n}$ converges weakly to the law of $X^{(k)}$ under $P_\rho$ as probability measures on the space $D(\mathbb{R}_+,M)$. In particular, the finite-dimensional distributions converge for any collection of times $t_1,\dots,t_m\geq 0$, $m\in\mathbb{N}$.}

\begin{proof} Fix $k$, and define $V_k$ as above Lemma \ref{muksupport}. Our first step is to characterise the Dirichlet form $(\mathcal{E}^{(k)},\mathcal{D}^{(k)})$ of the Markov chain $X^{(k)}$, which by Theorem \ref{trace} is given by (\ref{formtrace}), (\ref{domaintrace}) with $G=V_k$ and $\nu=\mu^{(k)}$. Since $F$ is compact, we have that $(\mathcal{E},\mathcal{D})=(\mathcal{E},\mathcal{F})$ (see \cite[p.\ 35]{Kig}), and so $(\mathcal{E},\mathcal{D})$ is recurrent. Hence we have from Lemma \ref{goodtrace} that $(\mathcal{E}^{(k)},\mathcal{D}^{(k)})$ is also a resistance form with associated resistance metric $R^{(k)}:=R|_{V_k\times V_k}$. In particular, we obtain that
\[\mathcal{E}^{(k)}(f,f)=\frac12\sum_{x,y\in V_k}c^{(k)}(x,y)(f(y)-f(x))^2,\]
where the conductances $(c^{(k)}(x,y))_{x,y\in V_k}$ are uniquely determined by the resistance $R^{(k)}$ \cite[Theorem 1.7]{Kigdendrite}.

We similarly have that the Dirichlet form $(\mathcal{E}^{n,k},\mathcal{D}^{n,k})$ of the Markov chain $X^{n,k}$ is given by
\[\mathcal{E}^{n,k}(f,f)=\frac12\sum_{x,y\in V_{n,k}}c^{n,k}(x,y)(f(y)-f(x))^2,\]
where $V_{n,k}:=\{x^n_i:\:i\in I_k\}$, and we note that for large $n$ we have that the cardinality of $V_{n,k}$ and $V_k$ are both equal. We will now check that
\begin{equation}\label{condlim}
\left(c^{n,k}(x_i^n,x_j^n)\right)_{i,j\in I_k}\rightarrow \left(c^{(k)}(x_i,x_j)\right)_{i,j\in I_k}.
\end{equation}
Observe that, from the definition of the resistance metric, we have $c^{n,k}(x_i^n,x_j^n)\leq R_n(x_i^n,x_j^n)^{-1}$. Hence we find that
\[\limsup_{n\rightarrow\infty}\max_{\substack{i,j\in I_k:\\i\neq j}}c^{n,k}(x_i^n,x_j^n)\leq
\max_{\substack{i,j\in I_k:\\i\neq j}}R(x_i,x_j)^{-1}<\infty.\]
In particular, for any subsequence $(c^{n_m,k}(x_i^{n_m},x_j^{n_m}))_{i,j\in I_k}$, we have a convergent subsubsequence $(c^{n_{m_l},k}(x_i^{n_{m_l}},x_j^{n_{m_l}}))_{i,j\in I_k}$ with limit $(\tilde{c}(x_i,x_j))_{i,j\in I_k}$. Define an associated form $(\tilde{\mathcal{E}},\tilde{\mathcal{D}})$ by setting
\[\tilde{\mathcal{E}}(f,f)=\frac12\sum_{x,y\in V_k}\tilde{c}(x,y)(f(y)-f(x))^2,\]
and $\tilde{\mathcal{D}}:=\{f:V_k\rightarrow \mathbb{R}\}$, and let $\tilde{R}$ be the associated resistance (which may \emph{a priori} be infinite between pairs of vertices). It is then an elementary exercise to check that $c^{n_{m_l},k}\rightarrow \tilde{c}$ implies $(R^{n_{m_l}}(x_i^{n_{m_l}},x_j^{n_{m_l}}))_{i,j\in I_k}\rightarrow (\tilde{R}(x_i,x_j))_{i,j\in I_k}$. However, we also know $(R^{n_{m_l}}(x_i^{n_{m_l}},x_j^{n_{m_l}}))_{i,j\in I_k}\rightarrow ({R}^{(k)}(x_i,x_j))_{i,j\in I_k}$, and so it must be the case that $\tilde{R}=R^{(k)}$. In turn, this implies $\tilde{c}=c^{(k)}$ (see \cite[Theorem 1.7]{Kigdendrite}), and the conclusion at (\ref{condlim}) follows as desired.

Next, note that for each $i\in I_k$
\begin{equation}\label{mconv}
\mu_n^{(k)}(\{x_i^n\})=\mu_n\left(K_i^k\right)\rightarrow \mu\left(K_i^k\right)=\mu^{(k)}(\{x_i\})>0,
\end{equation}
where we have applied that $\mu_n\rightarrow\mu$ weakly, and that $K_i^k$ is a continuity set for the limiting measure. The fact that the limit is strictly positive was proved in Lemma \ref{muksupport}. These observations will allow us to check convergence of the generators. Specifically, the generator of $X^{(k)}$ is given by
\[\Delta^{(k)}f(x_i)=\frac{1}{\mu^{(k)}(\{x_i\})}\sum_{j\in I_k}c^{(k)}(x_i,x_j)(f(x_j)-f(x_i)).\]
Similarly, if we define $\pi_{n,k}:V_{n,k}\rightarrow V_k$ by $x_i^n\mapsto x_i$ (which is a bijection for large $n$), then the generator of $\pi_{n,k}(X^{n,k})$ is given by
\[\Delta^{n,k}f(x_i)=\frac{1}{\mu_{n}^{(k)}(\{x_i^n\})}\sum_{j\in I_k}c^{n,k}(x_i^n,x_j^n)(f(x_j)-f(x_i)).\]
Hence, (\ref{condlim}) and (\ref{mconv}) imply that
\[\max_{i\in I_k}\left|\Delta^{(k)}f(x_i)-\Delta^{n,k}f(x_i)\right|\rightarrow 0\]
for any $f:V_k\rightarrow \mathbb{R}$. Since the starting points of the processes satisfy $\pi_{n,k}(X^{n,k}_0)= X^{(k)}_0=\rho$ (as local time accumulates immediately), this generator convergence is enough to establish the distributional convergence  $\pi_{n,k}(X^{n,k})\rightarrow X^{(k)}$  (see \cite[Theorem 19.25]{Kall}). To complete the proof of the first claim, it is thus enough to recall that $d_M(x_i^n,\pi_{n,k}(x_i^n))=d_M(x_i^n,x_i)\rightarrow 0$ for each $i$.

For the claim regarding finite-dimensional distributions, one notes that convergence in the space $D(\mathbb{R}_+,M)$ implies convergence of finite-dimensional distributions at times $t_1,\dots,t_m$ that are continuity times for the process $X^{(k)}$, i.e.\ times at which $X^{(k)}$ is continuous, $P_\rho$-a.s. Furthermore, it is elementary to check that every $t\geq 0$ is a continuity time for the finite state space continuous time Markov chain $X^{(k)}$.
\end{proof}

The remaining ingredient we need to establish the result of interest is the following lemma.

{\lemma\label{l3} The laws of $X^n$ under $P_{\rho_n}^n$, $n\geq 1$, form a tight sequence in $D(\mathbb{R}_+,M)$. Moreover, for any $\varepsilon>0$ and $t\geq 0$,
\begin{equation}\label{merging}
\lim_{k\rightarrow\infty}\limsup_{n\rightarrow\infty}P^{n}_{\rho_n}\left(R_n\left(X^n_t, X^{n,k}_t\right)>\varepsilon\right)=0.
\end{equation}}
\begin{proof} To verify tightness, it will suffice to check Aldous' tightness criteria (see, for example, \cite[Theorem 16.11]{Kall}): for any bounded sequence of $X^n$ stopping times $\sigma_n$ and any sequence $\delta_n\rightarrow 0$, it holds that, for $\varepsilon>0$, $P_{\rho_n}^n(R_n(X^n_{\sigma_n},X^n_{\sigma_n+\delta_n})>\varepsilon)$. Applying the strong Markov property, to establish this it will be enough to show that
\begin{equation}\label{tightness}
\sup_{x\in F_n}P_{x}^n\left(R_n(x,X^n_{\delta_n})>\varepsilon\right)\rightarrow 0.
\end{equation}
To do this, we note that the UVD condition implies the following exit time estimate
\begin{equation}\label{etest}
\sup_{x\in F_n}P^{n}_{x}
\left(\sup_{0\leq t\leq \delta}R_n\left(x,X^{n}_t\right)>\varepsilon\right)\leq c_1 e^{-\frac{c_2 \varepsilon}{v^{-1}(\delta/\varepsilon)}},
\end{equation}
uniformly in $n$, where $v$ is the function appearing in the definition of UVD (see \cite[Proposition 4.2 and Lemma 4.2]{Kum}). Moreover, the doubling property of $v$ implies that $v(r)\geq c_3r^{c_4}$ for $r\leq 1$, and so $v^{-1}(\delta/\varepsilon)\leq c_5(\delta/\varepsilon)^{c_6}$ for $\delta$ suitable small. The result at (\ref{tightness}) follows.

To prove (\ref{merging}), first note that
\begin{eqnarray*}
\left|A^{n,k}_t-t\right|&=&\left|\int_{F_n}L^n_t(x)\mu_n^{(k)}(dx)-\int_{F_n}L^n_t(x)\mu_n(dx)\right|\\
&\leq&\int_{F_n}\left|L^n_t(\phi^{n,k}(x))-L^n_t(x)\right|\mu_n(dx)\\
&\leq & \mu_n(F_n)\sup_{x\in F_n}\left|L^n_t(\phi^{n,k}(x))-L^n_t(x)\right|.
\end{eqnarray*}
Now, by (\ref{hconv}), $\mu_n(F_n)\rightarrow\mu(F)$, and the compactness of the space $(F,R)$ implies that the latter is a finite limit. Hence, also applying (\ref{closepoints}) and the local time equicontinuity result of Lemma \ref{ltcont}, it follows that
\[\lim_{k\rightarrow\infty}\limsup_{n\rightarrow\infty}
{P}^n_{\rho_n}\left(\sup_{0\leq t\leq T}\left|A^{n,k}_t-t\right|>\varepsilon\right)=0.\]
Taking inverses, we thus find that
\[\lim_{k\rightarrow\infty}\limsup_{n\rightarrow\infty}
{P}^n_{\rho_n}\left(\sup_{0\leq t\leq T}\left|{\tau^{n,k}(t)}-t\right|>\varepsilon\right)=0.\]
From this, we see that, for any $t,\varepsilon,\delta\geq 0$,
\begin{eqnarray*}
{\lim_{\delta\rightarrow0}\limsup_{n\rightarrow\infty}P^{n}_{\rho_n}
\left(R_n\left(X^n_t, X^{n,k}_t\right)>\varepsilon\right)}&\leq&\lim_{\delta\rightarrow0}\limsup_{n\rightarrow\infty}P^{n}_{\rho_n}
\left(\sup_{s\in[t-\delta,t+\delta]\cap\mathbb{R}_+}R_n\left(X^n_t, X^{n}_s\right)>\varepsilon\right)\\
&\leq&\lim_{\delta\rightarrow0}\limsup_{n\rightarrow\infty}\sup_{x\in F_n}
P^{n}_{x}
\left(\sup_{s\in[0,2\delta]}R_n\left(x, X^{n}_s\right)>\varepsilon\right)\\
&=&0,
\end{eqnarray*}
where to deduce the second inequality, we apply the Markov property at time $\max\{0,t-\delta\}$, and (\ref{etest}) to deduce the equality.
\end{proof}

Piecing together Lemmas \ref{l1} and \ref{l2}, and (\ref{merging}), we obtain that the finite-dimensional distributions of $X^n$ converge to those of $X$ for any collection of times $t_1,\dots,t_m\in T_X$, where we recall that $T_X$ is the set of continuity times of $X$ (see \cite[Theorem 4.28]{Kall}). Together with the tightness of $X^n$, as established in Lemma \ref{l3}, we arrive at the desired conclusion by applying \cite[Theorem 16.10]{Kall}.

{\propn \label{compactcase} The law of $X^{n}$ under $P^n_{\rho_n}$ converges weakly to the law of $X$ under $P_\rho$ as probability measures on the space $D(\mathbb{R}_+,M)$.}

\subsection{Locally compact case}

In this section, we explain how to extend from the compact case to the locally compact case. The proof will involve considering the trace of the relevant processes on bounded subsets (cf.\ the proof of \cite[Theorem 1.4]{BCK}). Key to this approach is the following lemma, which is an immediate consequence of Lemma \ref{goodtrace}. (Recall that we are assuming $(\mathcal{E},\mathcal{D})$ is recurrent for $(F,R,\mu,\rho)\in\mathbb{F}$.)

{\lemma\label{rtrace} Let $(F,R,\mu,\rho)\in\mathbb{F}$. For $r\geq 0$, let $(\mathcal{E}^{(r)},\mathcal{D}^{(r)})$ be the trace of $(\mathcal{E},\mathcal{D})$ on $F^{(r)}$ with respect to the measure $\mu^{(r)}$. Then $(\mathcal{E}^{(r)},\mathcal{D}^{(r)})$ is a resistance form on $F^{(r)}$ with associated resistance metric $R^{(r)}$.}
\bigskip

A second key ingredient for our argument is the following uniform exit time estimate for sequences of resistance forms satisfying UVD.

{\lemma \label{exittimes} Suppose $(F_n,R_n,\mu_n,\rho_n)_{n\geq 1}$ is a sequence in $\mathbb{F}$ satisfying UVD, then, for any $T<\infty$,
\[\lim_{r\rightarrow\infty}\sup_{n\geq 1}P_{\rho_n}^n\left(\sup_{0\leq t\leq T}R_n(\rho_n,X_t^n)> r\right)=0.\]}
\begin{proof} Similarly to (\ref{etest}), we have by \cite[Proposition 4.2 and Lemma 4.2]{Kum} that
\[\sup_{n\geq 1}P^{n}_{\rho_n}
\left(\sup_{0\leq t\leq T}R_n\left(x,X^{n}_t\right)>r\right)\leq c_1 e^{-\frac{c_2 r}{v^{-1}(T/r)}}.\]
Letting $r\rightarrow\infty$ establishes the result.
\end{proof}

We are now ready to prove the main result of this section, which establishes the first claim of Theorem \ref{main1}.

{\propn \label{procconv} Suppose $(F_n,R_n,\mu_n,\rho_n)$, $n\geq 1$, and $(F,R,\mu,\rho)$ satisfy Assumption \ref{a1}. It is then possible to embed $(F_n,R_n)$, $n\geq 1$, and $(F,R)$ isometrically into the same metric space $(M,d_M)$ in such a way that the law of $X^n$ under $P^n_{\rho_n}$ converges weakly to the law of $X$ under $P_{\rho}$ as probability measures on the space $D(\mathbb{R}_+,M)$.}

\begin{proof} Under the assumption of the proposition, it is possible to suppose all the objects of the discussion have been isometrically embedded into a common metric space $(M,d_M)$ in the way described in Lemma \ref{embeddings}. Define $(\mathcal{E}^{(r)},\mathcal{D}^{(r)})$ as in the statement of Lemma \ref{rtrace}. By Theorem \ref{trace}, we have that this is a regular Dirichlet form on $L^2(F^{(r)}, \mu^{(r)})$, and the associated process $X^{(r)}$ is given by a time-change according to the additive functional
\[A^{(r)}_t:=\int_FL_t(x)\mu^{(r)}(dx).\]
By monotonicity and the fact that the various additive functionals are increasing, we have that $A^{(r)}_t\rightarrow \int_FL_t(x)\mu(dx) =t$ uniformly on compact time intervals, $P_\rho$-a.s. Similarly to the proof of Lemma \ref{l1}, it follows that if $t_1,\dots,t_m\in T_X$ for any $m\in\mathbb{N}$ (where again we denote by $T_X$ the continuity times of $X$), then $P_\rho$-a.s., $d_M(X^{(r)}_{t_i},X_{t_i})\rightarrow 0$ for each $i=1,\dots,m$. Moreover, writing $\tau^{(r)}$ for the right-continuous inverse of $A^{(r)}$, we have that for any bounded sequence of $X^{(r)}$ stopping times $\sigma_r$ and any sequence $\delta_r\rightarrow0$, it holds that, for any $\varepsilon,\delta>0$,
\begin{eqnarray*}
P_\rho\left(R\left(X^{(r)}_{\sigma_r},X^{(r)}_{\sigma_r+\delta_r}\right)>\varepsilon\right)
&=&P_\rho\left(R\left(X_{\tau^{(r)}(\sigma_r)},X_{\tau^{(r)}(\sigma_r+\delta_r)}\right)>\varepsilon\right)\\
&\leq&\sup_{x\in F}P_x\left(\sup_{s\leq \delta}R\left(x,X_s\right)>\varepsilon\right)+o(1),
\end{eqnarray*}
as $r\rightarrow\infty$. (We note that $\tau^{(r)}(\sigma_r)$ is a stopping time for $X$.) Since by (\ref{uvdforlimit}) we know that the limiting space $(F,R,\mu,\rho)$ satisfies uniform volume doubling, we can again apply \cite[Proposition 4.2 and Lemma 4.2]{Kum} as at (\ref{etest}) to deduce that the probability above is bounded by $c_1e^{-c_2\varepsilon/v^{-1}(\delta/\varepsilon)}$. Letting $\delta\rightarrow 0$, we obtain that Aldous' tightness criteria holds (cf.\ the proof of Lemma \ref{l3}), and so the laws of $X^{(r)}$ under $P_\rho$ are tight in $D(\mathbb{R}_+,M)$. Combining this with the above convergence of finite dimensional distributions, we obtain that, under $P_\rho$, $X^{(r)}$ converges to $X$ in distribution in the space $D(\mathbb{R}_+,M)$.

Next, let $X^{n,r}$ be the trace of $X^n$ on $F_n^{(r)}$ with respect to $\mu_n^{(r)}$. By Lemma \ref{rtrace}, the Dirichlet form of this process, $(\mathcal{E}_n^{(r)},\mathcal{D}_n^{(r)})$ say, is actually a resistance form with associated resistance metric $R_n^{(r)}$, cf.\ the corresponding result in the limiting case. Hence, recalling we have embedded all the relevant objects into $M$ in the way described by Lemma \ref{embeddings}, Proposition \ref{compactcase} yields that, for Lebesgue-almost-every $r\geq 0$, we have that the law of $X^{n,r}$ under $P^n_{\rho_n}$ converges weakly to the law of $X^{(r)}$ under $P_\rho$ as probability measures on the space $D(\mathbb{R}_+,M)$.

Finally, we observe that if $\sup_{0\leq t\leq T+1}R_n(\rho_n,X^n_t)\leq r$, then the time-change functional describing $X^{n,r}$ satisfies
\[A^{n,r}_t=\int_{F_n}L_t^n(x)\mu_n^{(r)}(dx) =\int_{F_n}L_t^n(x)\mu_n(dx)=t\]
for $t\leq T+1$. It follows that $X^{n,r}_t=X^n_t$ for $t\leq T$. Thus we find that, for any $\varepsilon>0$,
\[P_{\rho_n}^n\left(\sup_{0\leq t\leq T}R_n(X^n_t,X^{n,r}_t)>\varepsilon\right)
\leq P_{\rho_n}^n\left(\sup_{0\leq t\leq T+1}R_n(\rho_n,X^n_t)> r\right).\]
By Lemma \ref{exittimes}, this converges to $0$ as $r\rightarrow\infty$, uniformly in $n\geq 1$. Combining this with the conclusions of the previous two paragraphs completes the result.
\end{proof}

\subsection{Convergence of local times}\label{ltconvsec}

Again we suppose that the spaces $(F_n,R_n)$, $n\geq 1$, and $(F,R)$ are isometrically embedded into a common metric space $(M,d_M)$ in such a way that the conclusion of Lemma \ref{embeddings} holds. Given the convergence result of Proposition \ref{procconv} (and \cite[Theorem 4.30]{Kall}, for example), it is further possible to suppose that $X^n$ started from $\rho_n$ and $X$ started from $\rho$ are coupled so that $X^n\rightarrow X$ in $D(\mathbb{R}_+,M)$, almost-surely. We will suppose that this is the case throughout this section, and write the joint probability measure as $P$. To prove the finite dimensional convergence of local times as at (\ref{ltconv1}), we will follow an approximation argument, based on averaging over small balls. To this end, it is useful to introduce the following functions: for $x\in M$, $\delta>0$,
\[f_{\delta,x}(y):=\max\left\{0,\delta-d_M(x,y)\right\}.\]
An immediate consequence of the continuity of local times of $X$ is the following lemma.

{\lemma \label{ltl1} $P$-a.s., for any $x\in F$ and $T\geq0$, as $\delta\rightarrow0$,
\[\sup_{t\in[0,T]}\left|\frac{\int_0^tf_{\delta,x}(X_s)ds}{\int_Ff_{\delta,x}(y)\mu(dy)}- L_t(x)\right|\rightarrow 0.\]}
\begin{proof} For the case when $F$ is compact, the result follows easily from Lemma \ref{ltcontlimit} (and the occupation density formula of \eqref{occdens}). In the case when $F$ is only locally compact, we note that on the event $\sup_{t\in[0,T]}R(\rho,X_s)\leq r$ it is the case that the local times of $X$ are identical to the local times of $X^{(r)}$ up to time $T$. Since the latter are continuous functions for each $t>0$, almost-surely, then so are the local times of $X$ for $t\in[0,T]$, almost-surely on  $\sup_{t\in[0,T]}R(\rho,X_s)\leq r$. Taking $r\rightarrow\infty$, and then $T\rightarrow\infty$, we deduce that the local times of $X$ are continuous functions for each $t>0$, almost-surely, and the result follows in this case as well.
\end{proof}

{\lemma \label{ltl2}$P$-a.s., for any $x\in F$, $T\geq0$ and $\delta>0$, as $n\rightarrow\infty$,
\[\sup_{t\in[0,T]}\left|\frac{\int_0^tf_{\delta,x}(X^n_s)ds}{\int_{F_n}f_{\delta,x}(y)\mu_n(dy)}-\frac{\int_0^tf_{\delta,x}(X_s)ds}{\int_Ff_{\delta,x}(y)\mu(dy)}\right| \rightarrow 0.\]}
\begin{proof} Fix $x\in F$. It is then possible to choose $r$ such that $\bar{B}_M(x,\delta)\cap{F}\subseteq F^{(r)}$ for every $\delta<1$. Moreover, since our choice of embeddings satisfies the conclusions of Lemma \ref{embeddings}, we may further suppose that $\mu_n^{(r)}\rightarrow \mu^{(r)}$ weakly as probability measures on $M$. It follows that
\[{\int_{F_n}f_{\delta,x}(y)\mu_n(dy)}={\int_{M}f_{\delta,x}(y)\mu_n^{(r)}(dy)}
\rightarrow {\int_{M}f_{\delta,x}(y)\mu^{(r)}(dy)}={\int_{F}f_{\delta,x}(y)\mu(dy)}>0,\]
where the strict positivity of the limit is a simple consequence of the fact that $\mu$ has full support. Thus it remains to show that, for any $T\geq 0$,
\begin{equation}\label{c0}
\sup_{t\in [0,T]}\left|\int_0^tf_{\delta,x}(X^n_s)ds-\int_0^tf_{\delta,x}(X_s)ds\right|\rightarrow0.
\end{equation}
To begin with, suppose that $X$ is continuous at time $t$. It is then the case that, for each $n$, there exists a homeomorphism $\lambda_n:[0,t]\rightarrow[0,t]$ with $\lambda_n(0)=0$ such that
\begin{equation}\label{c1}
\sup_{s\in[0,t]}\left|s-\lambda_n(s)\right|\rightarrow 0,
\end{equation}
and also
\begin{equation}\label{c2}
\sup_{s\in[0,t]}d_M\left(X^n_{\lambda_n(s)},X_s\right)\rightarrow 0.
\end{equation}
Now,
\begin{eqnarray*}
\left|\int_0^tf_{\delta,x}(X^n_{s})ds-\int_0^tf_{\delta,x}(X_{s})ds\right|&\leq&
\int_0^t\left|f_{\delta,x}(X^n_{\lambda_n(s)})-f_{\delta,x}(X_{s})\right|d\lambda_n(s)\\
&&+\int_0^t\left|f_{\delta,x}(X_{\lambda^{-1}_n(s)})-f_{\delta,x}(X_{s})\right|ds.
\end{eqnarray*}
The first term in the upper bound here converges to zero by (\ref{c2}). As for the second term, we have from (\ref{c1}) that $d_M(X_{\lambda^{-1}_n(s)},X_s)\rightarrow 0$ whenever $X$ is continuous at $s$. Since the times at which $X$ is not continuous is at most countable, the dominated convergence theorem yields that the second term also converges to zero, thereby establishing the limit (\ref{c0}) pointwise at times at which $X$ is continuous. To extend to the full result is straightforward, using again that the times at which
$X$ is not continuous is countable, as well as the monotonicity and continuity of the limit.
\end{proof}

{\lemma\label{ltl3} For any $x\in F$ and $T\geq0$, if $x_n\in F_n$ is such that $d_M(x_n,x)\rightarrow 0$, then
\[\lim_{\delta\rightarrow 0}\limsup_{n\rightarrow\infty}P\left(\sup_{t\in[0,T]}\left|\frac{\int_0^tf_{\delta,x}(X^n_s)ds}{\int_{F_n}f_{\delta,x}(y)\mu_n(dy)}-L^n_t(x_n)\right|>\varepsilon\right)=0.\]}
\begin{proof} For large $n$, we have that by the occupation density formula \eqref{occdens}
\[\left|\frac{\int_0^tf_{\delta,x}(X^n_s)ds}{\int_{F_n}f_{\delta,x}(y)\mu_n(dy)}-L^n_t(x_n)\right|\leq
\sup_{y,z\in \bar{B}_n(x_n,2\delta)}\left|L^n_t(y)-L^n_t(z)\right|.\]
Thus if the sequence $(F_n,R_n,\mu_n,\rho_n)_{n\geq 1}$ satisfies $\sup_n r_\infty(n)<\infty$, then the result follows from the local time equicontinuity result of Lemma \ref{ltcont}. In the general case, it is possible to obtain the result by considering the restriction to bounded subsets as in the last part of the proof of Proposition \ref{procconv}.
\end{proof}

From Lemmas \ref{ltl1}, \ref{ltl2} and \ref{ltl3}, we deduce that for any $x\in F$ and $T\geq 0$, if $x_n\in F_n$ is such that $d_M(x_n,x)\rightarrow 0$, then $(L_t^n(x_n))_{t\in[0,T]}\rightarrow(L_t(x))_{t\in[0,T]}$ in $P$-probability in $C([0,T],\mathbb{R})$. This result immediately extends to finite collections of points, which is enough to establish (\ref{ltconv1}).

\subsection{Time-changed processes}

In this section, we prove Corollary \ref{maincor}, starting by showing convergence of the time-change additive functionals.

{\propn\label{atn} If Assumption \ref{a2} holds, then $(A^n_t)_{t\geq0}\rightarrow (A_t)_{t\geq0}$ in distribution in the space $C(\mathbb{R}_+,\mathbb{R})$, simultaneously with the convergence of processes $X^n\rightarrow X$ in $D(\mathbb{R}_+,\mathbb{R})$, where we assume that $X^n$ is started from $\rho_n$, and $X$ is started from $\rho$.
}
\begin{proof} We first prove the result in the case that the underlying spaces are compact, i.e.\ when $(F_n,R_n,\mu_n,\rho_n)\in\mathbb{F}_c$, $n\geq 1$, and $(F,R,\mu,\rho)\in\mathbb{F}_c$. Suppose all the objects are isometrically embedded into a common space in the way described at (\ref{hconv}), and let $(x_i)_{i\geq 1}$ be as in Section~\ref{compsec}. Moreover, for each $k$, define $(K_i^k)_{i=1}^k$ as in \eqref{eq:setdec}, but with each set chosen to be a continuity set for the measure $\nu$, rather than for $\mu$. Then, for any $T\geq 0$, we have from the continuity of local times (Lemma \ref{ltcontlimit}) that, $P_{\rho}$-a.s.,
\[\sup_{t\in[0,T]}\left|A_t-\sum_{i=1}^kL_t(x_i)\nu(K_i^k)\right|\leq \nu(F)\sup_{t\in[0,T]}\sup_{\substack{y,z\in F:\\R(y,z)\leq 4\varepsilon_k}}\left|L_t(y)-L_t(z)\right|\rightarrow 0.\]
Next, from (\ref{ltconv1}), we deduce that
\begin{equation*}
\left(\sum_{i=1}^kL^n_t(x_i^n)\nu^n(K_i^k)\right)_{t\geq0}\rightarrow \left(\sum_{i=1}^kL_t(x_i)\nu(K_i^k)\right)_{t\geq0}
\end{equation*}
in distribution in $C(\mathbb{R}_+,\mathbb{R})$, where $x_i^n$ are also chosen as in Section \ref{compsec}. Furthermore, for large $n$, we have that
\[\sup_{t\in[0,T]}\left|A_t^n-\sum_{i=1}^kL^n_t(x_i^n)\nu^n(K_i^k)\right|\leq \nu_n(F_n)\sup_{t\in[0,T]}\sup_{\substack{y,z\in F_n:\\R_n(y,z)\leq 4\varepsilon_k}}\left|L^n_t(y)-L^n_t(z)\right|.\]
Under $P_{\rho_n}^n$, this converges to zero in probability as $n\rightarrow\infty$ and then $k\rightarrow\infty$ by Lemma \ref{ltcont}. Noting that, from Theorem \ref{main1}, the convergence of local times at (\ref{ltconv1}) occurs simultaneously with the convergence of processes, the desired result follows.

For the general case, one again proceeds by considering the restriction to bounded subsets similarly to the proof of Proposition \ref{procconv}. For this, it is useful to note that it is enough to consider radii $r$ that are continuity sets for both $\mu$ and $\nu$, since the collections of points of discontinuity of the maps $r\mapsto \mu(B_R(\rho,r))$ and $r\mapsto \nu(B_R(\rho,r))$ are both countable.
\end{proof}

We next check the divergence of the additive functional $(A_t)_{t\geq0}$, as defined at (\ref{atdef}).

{\lemma For $(F,R,\mu,\rho)\in\mathbb{F}$, and $\nu$ a locally finite Borel regular measure on $(F,R)$ with $\nu(F)>0$, we have $A_t\rightarrow\infty$, $P_x$-a.s.\ for any $x\in F$.}
\begin{proof} First note that, by \cite[Theorem 5.2.16]{FukuChen}, we have that $(\mathcal{E},\mathcal{D})$ is an irreducible Dirichlet form (see \cite[Section 2.1]{FukuChen} for a definition).
Since $(\mathcal{E},\mathcal{D})$ is recurrent, we can apply \cite[Theorem 3.5.6(ii)]{FukuChen} to deduce that $P_x(\tau_y<\infty)=1$, for all $x,y\in F$. Moreover, by \cite[Theorem 3.6.5]{MR}, we have that
$E_x(\int_0^\infty e^{-t}dL_t(x))=u(x,x)>0$, for all $x\in F$, where $(u(x,y))_{x,y\in F}$ is the potential density of $X$, as defined at (\ref{udef}). Combining these two observations, following the proof of \cite[Lemma 2.3]{Croyrogt} allows us to deduce that $\lim_{t\rightarrow\infty}\inf_{x\in F^{(r)}}L_t(x)=\infty$ for any $r\geq0$, $P_y$-a.s.\ for any $y\in F$. This readily yields the result.
\end{proof}

Note that the previous lemma implies that $\tau(t):=\inf\{s>0:A_s>t\}$ remains finite for all $t\geq 0$, and so confirms that $X_{\tau(t)}$ has an infinite lifetime. We are now in a position to complete the proof of Corollary \ref{maincor}.

\begin{proof}[Proof of Corollary \ref{maincor}] First, suppose we are in case (a); in particular, $\nu(F)$ has full support. Moreover, suppose that we have embedded all the objects of the discussion into a common metric space $(M,d_M)$ in the way described by Lemma \ref{embeddings}, and that the various processes are coupled so that $X^n\rightarrow X$ in $D(\mathbb{R}_+,M)$, and $A^n\rightarrow A$ in $D(\mathbb{R}_+,\mathbb{R}_+)$, almost-surely. As in Section \ref{ltconvsec}, denote the probability measure corresponding to the coupling by $P$. Now, note that, $P$-a.s., for any $t,\delta>0$ we have that $\int_{F}(L_{t+\delta}(x)-L_t(x))\mu(dx)=\delta>0$, and so, applying the continuity of local times, we can find an $\varepsilon>0$ such that $L_{t+\delta}(x)-L_t(x)\geq\varepsilon$ on a non-empty open set. Since $\nu(F)$ has full support, it readily follows that $(A_t)_{t\geq 0}$ is strictly increasing, $P$-a.s. Thus we can apply \cite[Theorem 7.2]{Whittpaper}, to deduce that $\tau^n\rightarrow \tau$ in $D(\mathbb{R}_+,\mathbb{R}_+)$, where the limiting function is strictly increasing and continuous, $P$-a.s. (Recall that $\tau^n$ is the right-continuous inverse of $A^n$, and $\tau$ is the right-continuous inverse of $A$.) Together with the convergence $X^n\rightarrow X$, this implies (see \cite[Theorem 3.1]{Whittpaper}) that $X^{n,\nu_n}\rightarrow X^\nu$ in $D(\mathbb{R}_+,M)$, $P$-a.s., which confirms the result.

The proof of part (b) is essentially the same, but involves different topologies. In particular, from $A^n\rightarrow A$ in $D(\mathbb{R}_+,\mathbb{R}_+)$, it is only possible in general to suppose $\tau^n\rightarrow \tau$ with respect to the Skorohod $M_1$ topology \cite[Theorem 7.1]{Whittpaper}. Given this convergence holds simultaneously with $X^n\rightarrow X$ in $D(\mathbb{R}_+,M)$, where $X$ is assumed to be continuous, we can apply the straightforward generalisation of \cite[Lemma A.6]{CroyMuir} to deduce the result.
\end{proof}

\section{Liouville Brownian motion}

Given an element in $(F,R,\mu,\rho)\in \mathbb{F}$, the associated Liouville Brownian motion is the process $X^\nu$, defined as at (\ref{xnudef}), where $\nu$ is the Liouville measure. To define this, let us first introduce the Gaussian free field on $F$, $(\gamma(x))_{x\in F}$ say, which we will suppose is pinned at $\rho$, and built on a probability space with probability measure $\mathbf{P}$ and expectation $\mathbf{E}$. In particular, we define $(\gamma(x))_{x\in F}$ to be a centred Gaussian field (i.e.\ $\mathbf{E}\gamma(x)=0$ for all $x\in F$), with covariances given by
\[{\rm{Cov}}(\gamma(x),\gamma(y))=g(x,y),\qquad \forall x,y\in F,\]
where $g(x,y)$ is the Green's function of $X$ killed on hitting $\rho$ (cf.\ the notation $g_{\{\rho\}}$ in the proof of Lemma \ref{ltcont}). Note that these assumptions imply that $\gamma(\rho)=0$, $\mathbf{P}$-a.s., and yield that an alternative way to characterise the covariances is via the formula
\[\mathbf{E}\left(\left(\gamma(x)-\gamma(y)\right)^2\right)=R(x,y),\qquad \forall x,y\in F.\]
(To deduce the latter identity, it is useful to observe that $2g(x,y)=R(\rho,x)+R(\rho,y)-R(x,y)$, see \cite[Theorem 4.3]{Kig}.) Thus we have from standard estimates for Gaussian random variables that
\begin{equation}\label{tailforgff}
\mathbf{P}\left(\left|\gamma(x)-\gamma(y)\right|\geq \varepsilon\right)\leq 2 e^{-\frac{\varepsilon^2}{2R(x,y)}},
\end{equation}
and substituting this for the estimate (\ref{tailforlt}), one can follow the proof of Lemma \ref{ltcont} to deduce that, if $(F,R,\mu,\rho)$ satisfies the volume doubling estimates of (\ref{uvdforlimit}), then $(\gamma(x))_{x\in F}$ is a continuous function, $\mathbf{P}$-a.s. (To check this continuity property, one might alternatively note that (\ref{uvdforlimit}) yields an estimate for the size of a $\varepsilon$-cover of $F^{(r)}$ of the form $c_1\varepsilon^{-c_2}$, and from this the result is an application of \cite[Theorem 8.6]{MRa}, for example.) In this case, for $\kappa>0$ fixed, setting (similarly to (\ref{2dlmeasure}))
\[\nu(dx)=e^{\kappa\gamma(x)-\frac{\kappa^2}{2}\mathbf{E}(\gamma(x)^2)}\mu(dx),\]
yields a locally finite, Borel regular measure on $(F,R)$ of full support, $\mathbf{P}$-a.s. (Note also that this choice of normalisation yields $\mathbf{E}\nu(dx)=\mu(dx)$.) Thus, for $\mathbf{P}$-a.e.\ realisation of $\nu$, we can define $X^\nu$ by the procedure at (\ref{xnudef}). Since under $P_x$ the starting point of $X^\nu$ is $x$, the corresponding \emph{quenched law} of $X^\nu$ started from $x\in F$ is well-defined; we will denote this by $P^\nu_x$. Moreover, we can define the \emph{annealed law} of the Liouville Brownian motion $X^\nu$ by integrating out the Liouville measure, i.e.\
\begin{equation}\label{anndef}
\mathbb{P}^{\rm LBM}_x\left(\cdot\right):=\int P_x^\nu\left(\cdot\right)\mathbf{P}(d\nu).
\end{equation}

The principal aim of this section is to show that if we have $(F_n,R_n,\mu_n,\rho_n)\rightarrow (F,R,\mu,\rho)$ in $\mathbb{F}$ and the UVD property holds (i.e.\ Assumption \ref{a1} is satisfied), then the associated Liouville measures and Liouville Brownian motions converge. To this end, we start by noting the equicontinuity of the Gaussian free fields in the sequence, which we will denote by $(\gamma_n(x))_{x\in F_n}$, $n\geq 1$. As for the continuity of $\gamma$, the proof of this result is identical to that of the local time equicontinuity result of Lemma \ref{ltcont}, with \eqref{tailforlt} replaced by \eqref{tailforgff}, and so is omitted.

{\lemma\label{gammanequi} If $(F_n,R_n,\mu_n,\rho_n)_{n\geq 1}$ is a sequence in $\mathbb{F}$ satisfying UVD, then, for each $\varepsilon>0$ and $r>0$,
\[\lim_{\delta\rightarrow0}\sup_{n\geq 1}\mathbf{P}\left(\sup_{\substack{y,z\in F^{(r)}_n:\\R_n(y,z)\leq\delta}}{\left|\gamma_n(y)-\gamma_n(z)\right|}\geq \varepsilon\right)=0.\]}
\bigskip

We can now deduce convergence of Liouville measures under Assumption \ref{a1}; the following result can be interpreted as a distributional version of Assumption \ref{a2}. We write $\nu_n$ for the Liouville measure associated with $(F_n,R_n,\mu_n,\rho_n)$.

{\lemma\label{nunconv} Suppose Assumption \ref{a1} holds, and that $(F_n,R_n,\mu_n,\rho_n)$, $n\geq 1$, and $(F,R,\mu,\rho)$ are isometrically embedded into a common (complete, separable, locally compact) metric space $(M,d_M)$ so that the conclusion of Lemma \ref{embeddings} holds. It is then the case that $\nu_n\rightarrow \nu$ in distribution with respect to the vague topology for locally finite Borel measures on $(M,d_M)$.}
\begin{proof} By \cite[Theorem 16.16]{Kall}, it will suffice to show that
\begin{equation}\label{fint}
\int_Mf(x)\nu_n(dx)\rightarrow\int_Mf(x)\nu(dx)
\end{equation}
in distribution, for any non-negative, continuous, compactly supported $f:M\rightarrow \mathbb{R}$. For each such $f$, we note that the support of $f$ is contained in ${B}_M(\rho,r/2)$ for some $r>0$ for which \eqref{embedconv} holds. Moreover, under the assumptions of the lemma, we have that $(f(x)e^{\kappa\gamma(x)-\frac{\kappa^2}{2}\mathbf{E}(\gamma(x)^2)})_{x\in F}$ is a continuous function on $F$, $\mathbf{P}$-a.s., and Lemma \ref{gammanequi} implies the equicontinuity of the functions $(f(x)e^{\kappa\gamma_n(x)-\frac{\kappa^2}{2}\mathbf{E}(\gamma_n(x)^2)})_{x\in F_n}$. Consequently, the result at \eqref{fint} can be proved in the same way as Proposition \ref{atn} if we can show the analogue of (\ref{ltconv1}) in this setting, i.e.\ if $(x_i^n)_{i=1}^k$ in $F_n$, $n\geq 1$, are such that $d_M(x_i^n,x_i)\rightarrow 0$ for some $(x_i)_{i=1}^k$ in $F$, then it holds that
\[\left(\gamma_n\left(x_i^n\right)\right)_{i=1,\dots, k}\rightarrow \left(\gamma\left(x_i\right)\right)_{i=1,\dots, k},\]
in distribution in $\mathbb{R}^k$. However, this is straightforward, since all the random variables above are centred Gaussian random variables, and
\[\left({\rm{Cov}}(\gamma_n(x_i^n),\gamma(x_j^n))\right)_{i,j=1,\dots,k}\rightarrow\left({\rm{Cov}}(\gamma(x_i),\gamma(x_j))\right)_{i,j=1,\dots,k},\]
where to deduce the latter convergence, it is again helpful to apply the identity
$2g(x,y)=R(\rho,x)+R(\rho,y)-R(x,y)$ from \cite[Theorem 4.3]{Kig}.
\end{proof}

From this convergence of Liouville measures, a simple adaptation of Corollary \ref{maincor} yields the convergence of Liouville Brownian motion in this setting. In particular, by the separability of the space of locally finite Borel measures on $(M,d_M)$ under the vague topology, we can suppose $\nu_n$ and $\nu$ are coupled so that the conclusion of Lemma \ref{nunconv} holds $\mathbf{P}$-a.s.\ (see \cite[Theorem 4.30]{Kall}, for example). Under this coupling, the proof of Corollary \ref{maincor} yields the almost-sure convergence of quenched laws, i.e.\ $P_{\rho_n}^{n,\nu_n}\rightarrow P_{\rho}^{\nu}$, weakly as probability measures on $D(\mathbb{R}_+,M)$, $\mathbf{P}$-a.s., where, for $\nu_n$ given, $P_{\rho_n}^{n,\nu_n}$ is the law of $X^{n,\nu_n}$ started from $\rho_n$. Integrating out the above result with respect to $\mathbf{P}$ then gives the following, where $\mathbb{P}^{{\rm LBM}_n}_{\rho_n}$ is the annealed law of $X^{n,\nu_n}$ started from $\rho_n$.

{\propn\label{lbmconv} Suppose Assumption \ref{a1} holds. It is then possible to isometrically embed $(F_n,R_n,\mu_n,\rho_n)$, $n\geq 1$, and $(F,R,\mu,\rho)$ into a common metric space $(M,d_M)$ so that
\[\mathbb{P}^{{\rm LBM}_n}_{\rho_n}\rightarrow\mathbb{P}^{{\rm LBM}}_{\rho}\]
weakly as probability measures on $D(\mathbb{R}_+,M)$.}

{\rem {\rm It is natural to ask for heat kernel estimates for the Liouville Brownian motion $X^\nu$. In the two-dimensional setting, this is a significant challenge (see \cite{KajAnd, MRVZ} for work in this direction). Here, however, estimating the quenched heat kernel of $X^\nu$ is straightforward. Indeed, as noted above, the measure estimates for $\mu$ at (\ref{uvdforlimit}) imply that the density $\frac{d\nu}{d\mu}(x)=e^{\kappa\gamma(x)-\frac{\kappa^2}{2}\mathbf{E}(\gamma(x)^2)}$ is $\mathbf{P}$-a.s.\ a continuous, strictly positive function, and so uniformly bounded away from 0 and $\infty$ on compact regions. Thus, uniformly over compact regions, $\nu$ satisfies the same measure estimates as $\mu$ (up to constants that depend on the particular region and realisation of $\nu$). In particular, this implies that, up to constants, the short-time on-diagonal quenched heat-kernel behaviour of $X^\nu$ will be the same as for the original process $X$. (See \cite{Kum} for particular heat kernel estimates that hold under uniform volume doubling.)}}

{\example\label{lbmexamples} {\rm As simple examples, one might consider graphical approximations to tree-like and low-dimensional fractal spaces. In the following, we briefly introduce some of these.

\noindent
(i) The most basic example would be to set $F_n$ to be the integer lattice $\mathbb{Z}$ equipped with the rescaled Euclidean distance $R_n(x,y)=n^{-1}|x-y|$, counting measure $\mu_n(A):=n^{-1}|A|$, and distinguished vertex $\rho_n=0$. Then Assumption \ref{a1} is satisfied with limit $(F,R,\mu,\rho)$ given by $F=\mathbb{R}$, $R$ the Euclidean metric, $\mu$ one-dimensional Lebesgue measure, and $\rho=0$.

\noindent
(ii) More generally than the previous example, one might consider a family of graph trees $(G_n)_{n\geq1}$ for which there exist scaling factors $(a_n)_{n\geq 1}$, $(b_n)_{n\geq1}$ such that $(G_n,a_nR_n,b_n\mu_n,\rho_n)_{n\geq1}$ satisfies Assumption \ref{a1} for some limiting tree $(F,R,\mu,\rho)\in \mathbb{F}$, where: $R_n$ is the resistance metric associated with unit resistances along edges of $G_n$; $\mu_n$ is the counting measure on $G_n$; and $\rho_n$ is a distinguished vertex. (Here and in the following, we call $G$ a graph tree if it is connected and contains no cycle.) In this setting the resistance metric is identical to the usual shortest path graph distance, for which the assumptions are generally easier to check. For example, it is elementary to check the result for the graphs approximating the Vicsek set, as shown in Figure \ref{vicsek}, with $a_n=3^{-n}$, $b_n=5^{-n}$.

\begin{figure}[ht]
\begin{center}
\scalebox{0.06}{\includegraphics{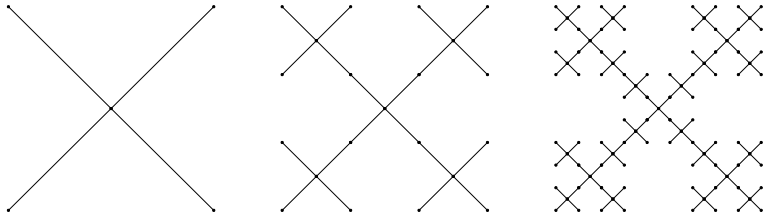}}
\end{center}
\vspace{-10pt}
\caption{The Vicsek set graphs $G_1$, $G_2$, $G_3$.}\label{vicsek}
\end{figure}

\noindent
(iii) It is known that the resistance metric on the graphs approximating nested fractals, again when unit resistors are placed along edges, can be rescaled to yield a resistance metric on the limiting fractal (this is an application of the homogenisation result of \cite[Theorem 3.8]{peir}, for example). In Section \ref{rcmfractalsec}, we introduce a more general class of fractals, so leave details until later (alternatively, see \cite{lind} for background on nested fractals). However, as an illustrative example, we note that the Assumption \ref{a1} applies to the sequence $(G_n,a_nR_n,b_n\mu_n,\rho_n)$, $n\geq 1$, where: $G_n$ is the $n$th level Sierpi\'nski gasket graph,
as introduced at the end of Section $1$ (see Figure \ref{sg}); $R_n$ is the associated resistance metric; $\mu_n$ is the counting measure on vertices; $\rho_n$ can be chosen arbitrarily as long as $(\rho_n)_{n\geq1}$ converges in $\mathbb{R}^2$; and the scaling factors are given by $a_n=(3/5)^n$, $b_n=3^{-n}$. Furthermore, we note that the results in this section also establish that
\[a_n^{1/2}\sup_{x\in G_n}\gamma_{G_n}(x)\rightarrow \sup_{x\in F}\gamma(x)\]
in distribution, where $\gamma_{G_n}$ is the Gaussian free field on $G_n$, and $\gamma$ the Gaussian free field  on the limiting fractal; this refines the result of \cite[Theorem 2.2]{KZ} for these examples.

\noindent
(iv) Finally, another fractal for which our present setting is appropriate is the two-dimensional Sierpi\'nski carpet and its graphical approximations, as shown in Figure \ref{sc}. (Again, the results would apply to other low-dimensional carpets.) Whilst the exact resistance scaling factor is not known in this case, previous results allow us to control the resistance in terms of the graph distance (cf.\ the comments in \cite[Section 5.4]{CroyLT}). It follows that there exist subsequences $(G_{n_i},a_{n_i}R_{n_i},b_{n_i}\mu_{n_i},\rho_{n_i})$ (where again $R_n$ is the resistance metric on $G_n$ with unit edge resistances, and $\mu_n$ is counting measure) that satisfy Assumption \ref{a1} with $a_n=\gamma^n$ for some $\gamma\in (0,1)$, and $b_n=8^{-n}$.

\begin{figure}[ht]
\begin{center}
\scalebox{0.06}{\includegraphics{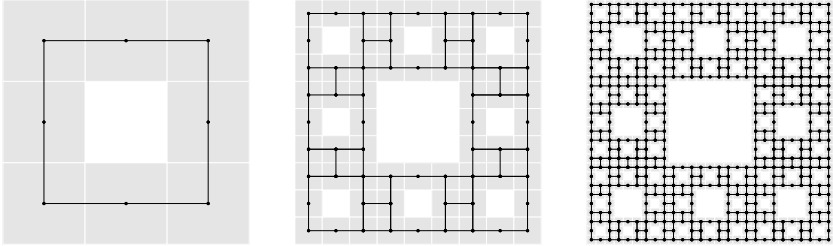}}
\end{center}
\vspace{-10pt}
\caption{The Sierpi\'nski carpet graphs $G_1$, $G_2$, $G_3$.}\label{sc}
\end{figure}}}

\section{Bouchaud trap model}\label{bouchsec}

We start this section by introducing the (symmetric) Bouchaud trap model (BTM) on a locally finite, connected graph $G=(V_G,E_G)$. To do this, we first introduce a trapping landscape $\xi=(\xi_x)_{x\in V_G}$, which is a collection of independent and identically distributed strictly-positive random variables built on a probability space with probability measure $\mathbf{P}$. Conditional on $\xi$, the dynamics of the BTM are then given by a continuous-time $V_G$-valued Markov chain $X^\xi=(X^\xi_t)_{t\geq 0}$ with jump rate from $x$ to $y$ given by $1/\xi_x$ if $\{x,y\}\in E_G$, and jump rate 0 otherwise. The quenched law of $X^\xi$ started from $x$ (i.e.\ the law given $\xi$) will be denoted $P_x^\xi$, and the corresponding annealed law, obtained by integrating out $\xi$, by $\mathbb{P}^{\rm BTM}_x$ (cf.\ (\ref{anndef})).

To put the BTM into the framework of this article, we note that it can be obtained as a time-change of the continuous time simple random walk on $G$. In particular, let $R_G$ be the resistance metric on $V_G$ obtained by placing unit resistors along edges, and $\mu_G$ be the counting measure on $V_G$ (i.e.\ $\mu_G(\{x\})=1$ for all $x\in V_G$). As in Section \ref{rfsec}, we can naturally associate a process $X$ with the triple $(V_G,R_G,\mu_G)$, and it is an elementary exercise to check that this is the continuous-time $V_G$-valued Markov chain with unit jump rate along edges. Moreover, by time-changing $X$ as at (\ref{xnudef}) according to the measure $\nu_G$ defined by setting $\nu_G(\{x\})=\xi_x$, we obtain $X^\xi$.

Similarly to the previous section, our goal is to present scaling limits of the BTM for sequences of recurrent graphs which, when equipped with resistance metrics and counting measure, satisfy UVD and converge in the Gromov-Hausdorff-vague topology to a limit in $\mathbb{F}$. We will do this in the case when the trapping environment is heavy-tailed. More specifically, in this section we make the following assumption.

{\assu \label{btmassu} Suppose $(G_n)_{n\geq 1}$ is a sequence of locally finite, connected graphs with vertex sets $V_n$, resistance metrics $R_n$ (as above, here we assume that individual edges have unit resistance), counting measures $\mu_n$, and distinguished vertices $\rho_n$. Suppose further that each graph $G_n$ is recurrent, so that $(V_n,R_n,\mu_n,\rho_n)\in \mathbb{F}$. Moreover, assume that there exist scaling factors $(a_n)_{n\geq 1}$, $(b_n)_{n\geq 1}$ such that $(V_n,a_nR_n,b_n\mu_n,\rho_n)_{n\geq 1}$ satisfy Assumption \ref{a1}, where the measure $\mu$ of the limiting space $(F,R,\mu,\rho)\in \mathbb{F}$ is non-atomic. Finally, we suppose that each $G_n$ is equipped with a trapping landscape $\xi^n=(\xi^n_x)_{x\in V_n}$ such that
\begin{equation}\label{alphatail}
\mathbf{P}\left(\xi^n_x>u\right)\sim u^{-\alpha}
\end{equation}
for some fixed $\alpha\in(0,1)$, where $f (u)\sim g (u)$ means $\lim_{u\to\infty}f(u)/g(u)=1$.}

{\rem {\rm It would be straightforward to replace (\ref{alphatail}) with the assumption that the random variables $\xi^n_x$ are in the domain of attraction of a stable law with index $\alpha$, but we do not do so here for reasons of brevity.}}
\bigskip

We next describe the limits of the trapping landscape and BTM under the above assumption. We will show that the former is given by the natural generalisation to (\ref{onedfinmeasure}) obtained by setting
\[\nu(dx):=\sum_{i}v_i\delta_{x_i}(dx),\]
where $(v_i,x_i)$ are the points of a Poisson process on $(0,\infty)\times F$ with intensity $\alpha v^{-1-\alpha}dv\mu(dx)$, and $\delta_{x}$ is the probability measure on $F$ that places all its mass at $x$; note that this is a locally finite, Borel regular measure on $(F,R)$ of full support, $\mathbf{P}$-a.s.\ (where we suppose $\nu$ is also built on the probability space with probability measure $\mathbf{P}$). Moreover, the latter is given by $X^{\nu}$, that is, the time-change of the process $X$ naturally associated with $(F,R,\mu)$ by the measure $\nu$. Reflecting the terminology for the corresponding one-dimensional object, we will call this the $\alpha$-FIN process on $(F,R,\mu)$.
In general, this is not a diffusion, but under our assumptions it will be whenever $X$ is, and in this case we will call it the $\alpha$-FIN diffusion. Given $\nu$, the quenched law of $X^\nu$ started from $x$ will be denoted by $P^{\nu}_x$, and the associated annealed law $\mathbb{P}^{\rm FIN}_x$.

The following lemma establishes convergence of the trapping landscapes. We write $\nu_n$ for the measure on $V_n$ induced by the trapping landscape $\xi^n$.

{\lemma\label{nunconv2} Suppose Assumption \ref{btmassu} holds, and $(V_n,a_nR_n,b_n\mu_n,\rho_n)$, $n\geq 1$, and $(F,R,\mu,\rho)$ are isometrically embedded into a common (complete, separable, locally compact) metric space $(M,d_M)$ so that the conclusion of Lemma \ref{embeddings} holds. It is then the case that $b_n^{1/\alpha}\nu_n\rightarrow \nu$ in distribution with respect to the vague topology for locally finite Borel measures on $(M,d_M)$.}
\begin{proof} By \cite[Theorem 16.16]{Kall} (and the fact that measures of disjoint sets are independent under both $\nu_n$ and $\nu$), it will suffice to show that $b_n^{1/\alpha}\nu_n(B)\rightarrow \nu(B)$ in distribution, for every relatively compact set $B\subseteq M$ such that $B$ is a continuity set for $\nu$, $\mathbf{P}$-a.s. Since $\nu(B)=0$, $\mathbf{P}$-a.s., if and only if $\mu(B)=0$, the latter requirement is equivalent to supposing $B$ is a continuity set for $\mu$. For such a $B$, we have by assumption that $b_n\mu_n(B)\rightarrow \mu(B)$. Hence we have by an elementary computation that
\begin{equation}\label{laplaceconv}
\mathbf{E}\left(e^{-\lambda b_n^{1/\alpha}\nu_n(B)}\right)=\mathbf{E}\left(e^{-\lambda b_n^{1/\alpha}\sum_{x\in B}\xi^n_x}\right)=\left(1-\lambda^\alpha b_n \Gamma(1-\alpha)+o(b_n)\right)^{\mu_n(B)}\rightarrow e^{-\lambda^\alpha \Gamma(1-\alpha)\mu(B)}.
\end{equation}
Moreover, it is a simple application of Campbell's theorem \cite[(3.17)]{Kingman} that
\[\mathbf{E}\left
(e^{-\lambda\nu(B)}\right)
=\mathbf{E}\left(e^{-\sum_{i:x_i\in B}\lambda v_i}\right)=e^{-\int_{(0,\infty)\times B}\left(1-e^{-\lambda v}\right)\alpha v^{-1-\alpha}dv\mu(dx)}=e^{-\lambda^\alpha\Gamma(1-\alpha)\mu(B)},
\]
and so we are done.
\end{proof}

In light of Lemma \ref{nunconv2}, and incorporating the scaling factors where appropriate, we can proceed exactly as for the proof of Proposition \ref{lbmconv} to deduce convergence of the rescaled BTMs. We write $\mathbb{P}^{{\rm BTM}_n}_x$ for the annealed law of the BTM $X^{n,\xi^n}$ on $G_n$.

{\propn\label{btmresult} Suppose Assumption \ref{btmassu} holds. It is then possible to isometrically embed $(V_n,a_nR_n,b_n\mu_n,\rho_n)$, $n\geq 1$, and $(F,R,\mu,\rho)$ into a common metric space $(M,d_M)$ so that
\[\mathbb{P}^{{\rm BTM}_n}_{\rho_n}\left(\left(X^{n,\xi^n}_{t/a_nb_n^{1/\alpha}}\right)_{t\geq 0}\in\cdot\right)
\rightarrow\mathbb{P}^{{\rm FIN}}_{\rho}\left(\left(X^{\nu}_{t}\right)_{t\geq 0}\in\cdot\right)\]
weakly as probability measures on $D(\mathbb{R}_+,M)$.}
\bigskip

{\example\label{SG55ex} {\rm As applications of Proposition \ref{btmresult}, we can consider the same spaces as discussed in Example \ref{lbmexamples}. For instance, for the BTM on the graph approximations to the Sierpi\'nski gasket $(G_n)_{n\geq1}$ of Example \ref{lbmexamples}(iii), we have the convergence of the annealed law of $(X^{n,\xi^n}_{3^{n/\alpha}(5/3)^nt})_{t\geq 0}$ to the annealed law of the $\alpha$-FIN diffusion on the Sierpi\'nski gasket.}}

\section{Random conductance model}\label{rcmsec}

We now recall from the introduction the random conductance model; this is defined similarly to the BTM, but with random weights now assigned to the edges rather than to the vertices. As in the previous section, let $G=(V_G,E_G)$ be a locally finite, connected graph. Let $\omega=(\omega_e)_{e\in E_G}$ be a collection of independent and identically distributed strictly-positive random variables built on a probability space with probability measure $\mathbf{P}$; these are the so-called \emph{random conductances}. (Actually, for our model of self-similar fractals, we will allow some local dependence.) Conditional on $\omega$, we define the \emph{variable speed random walk (VSRW)} $X^\omega=(X^\omega_t)_{t\geq 0}$ to be the continuous-time $V_G$-valued Markov chain with jump rate from $x$ to $y$ given by $\omega_{xy}$ if $\{x,y\}\in E_G$, and jump rate 0 otherwise. We obtain the associated \emph{constant speed random walk (CSRW)} $X^{\omega,\nu}=(X^{\omega,\nu}_t)_{t\geq 0}$ by setting the jump rate along edge $x$ to $y$ to be $\omega_{xy}/\nu(\{x\})$, where
\[\nu\left(\{x\}\right):=\sum_{e\in E_G:\:x\in e}\omega_e;\]
note that this is the time-change of $X^\omega$ according to the measure $\nu$, and has unit mean holding times at each vertex.

An important observation is that the VSRW and CSRW experience different trapping behaviour on edges of large conductance. In particular, if we have an edge of conductance $\omega_e\gg1$ (surrounded by other edges of conductance close to 1), then both the VSRW and CSRW cross the edge order $\omega_e$ times before escaping. However, each crossing only takes the VSRW a time of $1/\omega_e$, meaning that it is only trapped for a time of order 1, whereas each crossing for the CSRW takes a time of order 1, and so the latter process is trapped for a total time of order $\omega_e$. In particular, when the weights are bounded below, we might typically expect the VSRW associated with the conductances $\omega$ to behave like the VSRW on the unweighted graph, which in each of the examples we consider converges under scaling to Brownian motion on the limiting space. Moreover, we might expect the CSRW to behave like the Bouchaud trap model with trapping environment described by $\nu$, and therefore we expect to see FIN-type scaling limits for this process when the conductances are heavy-tailed.

The aim of this section is to make the heuristics of the previous paragraph rigourous, in the sense that we will show for the random conductance model on certain sequences of graphs that, if the weights are chosen to satisfy (similarly to \eqref{alphatail})
\begin{equation}\label{omegatail}
\mathbf{P}\left(\omega_e>u\right)\sim u^{-\alpha}
\end{equation}
for some fixed $\alpha\in(0,1)$, then the rescaled VSRW $X^\omega$ converges to the canonical Brownian motion on the limit space, and the rescaled CSRW $X^{\omega,\nu}$ converges to the $\alpha$-FIN diffusion. The two classes we discuss are graph trees, and a family of self-similar fractals.

\subsection{Random conductance model on trees}\label{rcmtreesec}

In this section, we will study the scaling limit of the VSRW and CSRW for the random conductance model on sequences of graph trees; our main result is Proposition \ref{rcmtreeresult}. As for the Bouchaud trap model, we will need to show that the associated time-change measures converge. The additional part of the argument will be to check that we also have homogenisation of the resistance metric when random conductances are placed along edges. In this setting, this is straightforward, since we can apply the law of large numbers along paths. We start by stating the main assumption of this section, which closely matches Assumption \ref{btmassu}. The restriction to compact spaces is only for convenience of presentation, and not essential.

{\assu \label{rcmtreeassu} Suppose $(T_n)_{n\geq 1}$ is a sequence of finite graph trees with vertex sets $V_n$, edge sets $E_n$, resistance metrics $R_n$ (here we assume that individual edges have unit resistance), counting measures $\mu_n$, and distinguished vertices $\rho_n$. In particular, $(V_n,R_n,\mu_n,\rho_n)\in \mathbb{F}_c$. Moreover, assume that there exist scaling factors $(a_n)_{n\geq 1}$, $(b_n)_{n\geq 1}$ such that $\sum_{n\geq 1}a_n^2<\infty$ and $(V_n,a_nR_n,b_n\mu_n,\rho_n)_{n\geq 1}$ satisfy Assumption \ref{a1}, where the limit space $(F,R,\mu,\rho)$ is in $\mathbb{F}_c$, and the measure $\mu$ is non-atomic. Finally, we suppose that each $T_n$ is equipped with random conductances $\omega^n=(\omega^n_e)_{e\in E_n}$ such that (\ref{omegatail}) holds.}
\bigskip

We start by considering the resistance metrics on the weighted graph trees. In particular, given the conductances $\omega^n$, we define $R_n^\omega$ to be the associated resistance metric on $V_n$. In the following lemma, we show that, for large $n$, these random metrics are uniformly close to a scaled copy of $R_n$. The scaling factor is given by $\varrho:=\mathbf{E}\omega_e^{-1}$.

{\lemma \label{resconv} Suppose Assumption \ref{rcmtreeassu} holds. It is then the case that, $\mathbf{P}$-a.s.,
\[\sup_{x,y\in V_n}a_n\left|R_n^\omega(x,y)-\varrho R_n(x,y)\right|\rightarrow 0.\]}
\begin{proof} Suppose $(V_n,a_nR_n)$ and $(F,R)$ are embedded into the same space $(M,d_M)$ such that \eqref{hconv} holds. Define $(x_i^n)_{i,n\geq1}$ and $(x_i)_{i\geq1}$ as in Section \ref{compsec}, so that $a_nR_n(x_i^n,x_j^n)\rightarrow R(x_i,x_j)$, for all $i,j\geq1$. Since $R_n^\omega(x_i^n,x_j^n)$ is the sum of $(\omega^n_e)^{-1}$ along the $R_n(x_i^n,x_j^n)$ edges in the path from $x_i^n$ to $x_j^n$, we obtain from (a fourth moment version of) the strong law of large numbers that, $\mathbf{P}$-a.s., ${R^\omega_n(x_i^n,x_j^n)}/{R_n(x_i^n,x_j^n)}\rightarrow \varrho$, for every $i,j\geq1$ (it is for this that the assumption $\sum_{n\geq 1}a_n^2<\infty$ is needed). In particular, the combination of the two previous observations implies that,  $\mathbf{P}$-a.s.,
\[\sup_{i,j\leq k}a_n\left|R_n^\omega(x_i^n,x_j^n))-\varrho R_n(x_i^n,x_j^n))\right|\rightarrow 0,\qquad \forall k\geq1.\]
Since the resistances of unit edges satisfy $(\omega^n_e)^{-1}\leq 1$, we also have that $R_n^{\omega}\leq R_n$. It thus follows that,  $\mathbf{P}$-a.s.,
\begin{eqnarray*}
\lefteqn{\limsup_{n\rightarrow\infty}\sup_{x,y\in V_n}a_n\left|R_n^\omega(x,y)-\varrho R_n(x,y)\right|}\\
&\leq&
\limsup_{n\rightarrow\infty}\left\{\sup_{i,j\leq k}a_n\left|R_n^\omega(x_i^n,x_j^n)-\varrho R_n(x_i^n,x_j^n)\right|+2\sup_{x\in V_n}\inf_{i\leq k}a_n R_n(x,x_i^n)\right\}\\
&\leq&2\varepsilon_k,
\end{eqnarray*}
where $\varepsilon_k$ is defined as in Section \ref{compsec}. In particular, since $\varepsilon_k\rightarrow0$ as $k\rightarrow\infty$, the result follows.
\end{proof}

Similarly to Lemma \ref{nunconv2}, we next check convergence of the measures $\nu_n$, where we define $\nu_n(\{x\})=\sum_{e\in E_n:\:x\in e}\omega_e^n$ for $x\in V_n$. The limiting measure $\nu$ is the FIN measure on $F$, defined as in the previous section.

{\lemma\label{nunconv3} Suppose Assumption \ref{rcmtreeassu} holds, and $(V_n,a_nR_n,b_n\mu_n,\rho_n)$, $n\geq 1$, and $(F,R,\mu,\rho)$ are isometrically embedded into a common (complete, separable, locally compact) metric space $(M,d_M)$ so that the conclusion of Lemma \ref{embeddings} holds. It is then the case that $2^{-1}b_n^{1/\alpha}\nu_n\rightarrow \nu$ in distribution with respect to the vague topology for locally finite Borel measures on $(M,d_M)$.}
\begin{proof} We first note that, if $\tilde{\mu}_{n}$ is a measure on $V_n$ defined by setting $\tilde{\mu}_{n}(\{x\})={\rm deg}_n(x)$, i.e.\ the usual graph degree of $x$ in $T_n$, then it is an elementary exercise to check that $d_{T_n}^P(\tilde{\mu}_{n},2{\mu}_{n})\leq2$, where $d_{T_n}^P$ is the Prohorov metric for measures on $T_n$. In particular, it follows that $b_n\tilde{\mu}_{n}\rightarrow 2\mu$ weakly as measures on $M$.

We next show that, for all $x\in M$, $r>0$ such that $B_M(x,r)$ is a continuity set for $\mu$,
\begin{equation}\label{ballmeasconv}
2^{-1}b_n^{1/\alpha}\nu_n\left(B_M(x,r)\right)\rightarrow \nu\left(B_M(x,r)\right)
\end{equation}
in distribution. Writing $B=B_M(x,r)$, we have that
\[\nu_n\left(B\right)=2\sum_{e\in E_n: e\subseteq B}\omega_e+\sum_{e\in E_n:\: e\cap B\neq \emptyset,\:e\not\subseteq B}\omega_e.\]
If we denote by $E_n^1(B)$ and $E_n^2(B)$ the subsets over which the two sums are taken, respectively, then we claim that
\begin{equation}\label{enlimits}
b_n\left|E_n^1(B)\right|\rightarrow \mu(B),\qquad b_n\left|E_n^2(B)\right|\rightarrow 0.
\end{equation}
Indeed, for the second limit, we note that the edges in $E_n^2(B)$ each connect to a distinct vertex in the annulus $\bar{B}_M(x,r+a_n)\backslash B$. It follows that, for any $\varepsilon>0$,
\[\limsup_{n\rightarrow\infty}b_n\left|E_n^2(B)\right|\leq \limsup_{n\rightarrow\infty} b_n\mu_n\left(\bar{B}_M(x,r+a_n)\backslash B\right)\leq \mu\left(\bar{B}_M(x,r+\varepsilon)\backslash B\right).\]
Since $B$ is a continuity set for $\mu$, the right-hand side can be made arbitrarily small by adjusting $\varepsilon$ as appropriate, which confirms the desired result. Given this, the first limit at (\ref{enlimits}) is a simple consequence of the identity $\tilde{\mu}_n(B)=2|E_n^1(B)|+|E_n^1(B)|$, and the conclusion of the first paragraph. Thus, exactly as for (\ref{laplaceconv}), we have that $\mathbf{E}(e^{-\lambda 2^{-1} b_n^{1/\alpha}\nu_n(B)})\rightarrow e^{-\lambda^\alpha \Gamma(1-\alpha)\mu(B)}$, which establishes (\ref{ballmeasconv}).

With the same techniques, it is straightforward to extend (\ref{ballmeasconv}) to the result that
\[\left(2^{-1}b_n^{1/\alpha}\nu_n\left(B_i\right)\right)_{i=1}^k\rightarrow \left(\nu\left(B_i\right)\right)_{i=1}^k\]
in distribution, where each set $B_i$ is a finite unions of balls that are continuity sets for $\mu$. In particular, since the collection of such sets forms a separating class (see \cite[p.\ 317]{Kall}), this implies the result (see \cite[Theorem 16.16 and Exercise 16.11]{Kall}).
\end{proof}

From Lemmas \ref{resconv} and \ref{nunconv3}, we are able to prove the main result of this section. We write $P^{{\rm VSRW}_n}_x$ for the quenched law of the VSRW $X^{n,\omega}$ on the tree $T_n$ with conductances $\omega^n$, started from $x$. We write $\mathbb{P}^{{\rm CSRW}_n}_{x}$ for the annealed law of the corresponding CSRW $X^{n,\omega,\nu}$. We write $P_x$ for the law of the Brownian motion on $(F,R,\mu)$ started from $x$, and $\mathbb{P}^{{\rm FIN}}_{\rho}$ is the annealed law of the associated $\alpha$-FIN diffusion, defined as in Section \ref{bouchsec}.

{\propn\label{rcmtreeresult} Suppose Assumption \ref{rcmtreeassu} holds. It is then possible to isometrically embed $(V_n,a_nR_n,b_n\mu_n,\rho_n)$, $n\geq 1$, and $(F,R,\mu,\rho)$ into a common metric space $(M,d_M)$ so that, $\mathbf{P}$-a.s.,
\begin{equation}\label{vsrwresult}
P^{{\rm VSRW}_n}_{\rho_n}\left(\left(X^{n,\omega}_{\varrho t/ a_nb_n}\right)_{t\geq 0}\in\cdot\right)\rightarrow P_\rho\left(\left(X_{t}\right)_{t\geq 0}\in\cdot\right)
\end{equation}
weakly as probability measures on $D(\mathbb{R}_+,M)$. Moreover,
\[\mathbb{P}^{{\rm CSRW}_n}_{\rho_n}\left(\left(X^{n,\omega,\nu}_{2\varrho t/ a_nb_n^{1/\alpha}}\right)_{t\geq 0}\in\cdot\right)
\rightarrow\mathbb{P}^{{\rm FIN}}_{\rho}\left(\left(X^{\nu}_{t}\right)_{t\geq 0}\in\cdot\right)\]
weakly as probability measures on $D(\mathbb{R}_+,M)$.}

\begin{proof} The proof is essentially the same as that of Theorem \ref{main1} and Corollary \ref{maincor}, but with care needed as the random metric $R_n^\omega$ is different to the metric $R_n$ used for the embedding. (We suppose throughout that the embeddings of $(V_n,a_nR_n,b_n\mu_n,\rho_n)$, $n\geq 1$, and $(F,R,\mu,\rho)$ into $(M,d_M)$ satisfy the conclusion of Lemma \ref{embeddings}.)

We first note that, since $R_n^\omega\leq R_n$ and Lemma \ref{resconv} holds, we have from the UVD assumption for the underlying space that, $\mathbf{P}$-a.s.,
\[b_n\mu_n\left(B_{n}^\omega(x,r)\right)\geq c_1v(r),\qquad\forall x\in V_n,\:r\in[r^\omega_0(n),r^\omega_\infty(n)+1],\]
and, for every $\varepsilon>0$,
\[b_n\mu_n\left(B_{n}^\omega(x,r)\right)\leq c_2v(r),\qquad\forall x\in V_n,\:r\in[\max\{r^\omega_0(n),a_n^{-1}\varepsilon\},r^\omega_\infty(n)+1],\]
where distances are defined with respect to the metric $R_n^\omega$ (note the truncation at $a_n^{-1}\varepsilon$ in the upper bound). These bounds are enough to repeat the proof of Lemma \ref{ltcont} (cf.\ the weaker version of UVD in \cite{CroyLT}) to deduce the equicontinuity of the rescaled local times $(a_nL^{n,\omega}_{\varrho t/a_nb_n}(x))_{x\in V_n}$ of the VSRW $X^{n,\omega}$ with respect to the distance $a_nR_n^\omega$, and, by Lemma \ref{resconv} again, the equicontinuity of these local times with respect to $a_nR_n$. We also claim that the above volume bounds yield that in place of (\ref{etest}) we have, for $\delta,\tilde{\varepsilon}>0$,
\begin{equation}\label{exptailest}
\limsup_{n\rightarrow\infty}\sup_{x\in V_n}P^{{\rm VSRW}_n}_{x}
\left(\sup_{0\leq t\leq \delta}a_nR_n\left(x,X^{n,\omega}_{\varrho t/ a_nb_n}\right)>\tilde{\varepsilon}\right)\leq c_1 e^{-\frac{c_2 \tilde{\varepsilon}}{v^{-1}(\delta/\tilde{\varepsilon})}},\qquad \mathbf{P}\mbox{-a.s.}
\end{equation}
Checking this requires only a minor adaptation of results from \cite{Kum}. Indeed, writing $\tau^{n,\omega}(x,r):=\inf\{t>0:\:a_nR_n(x,X^{n,\omega}_{\varrho t/ a_nb_n})> r\}$ and $h(r) = rv(r)$, the proof of \cite[Proposition 4.2]{Kum} gives the existence of constants $c_1, c_2$ such that
\[E^{{\rm VSRW}_n}_{x}\left(\tau^{n,\omega}(y,r)\right)\leq c_1h(r), \qquad E^{{\rm VSRW}_n}_{x}\left(\tau^{n,\omega}(x,r)\right)\geq c_2h(r),\]
for all $x,y\in V_n$, $r\in [\max\{a_nr^\omega_0(n),\varepsilon\},a_n(r^\omega_\infty(n)+1)]$ and $n\geq 1$, and from this it readily follows that
\[P^{{\rm VSRW}_n}_{x}\left(\tau^{n,\omega}(x,r)\leq s\right)\leq 1-c_3+\frac{c_4s}{h(r)}\]
for every $x\in V_n$, $r\in [\max\{a_nr^\omega_0(n),\varepsilon\},a_n(r^\omega_\infty(n)+1)]$, $s\geq 0$ and $n\geq 1$, cf.\ proof of \cite[Lemma 4.2]{Kum}. To obtain the exponential estimate of \eqref{exptailest}, we then follow the chaining argument of \cite[Lemma 4.2]{Kum}. This requires us to apply the previous exit time tail estimate for radii no smaller than $c_5v^{-1}(\tilde{\varepsilon}/\delta)$ (with respect to the metric $a_nR_n$). Noting that $a_nr^\omega_0(n)\rightarrow 0$, $\mathbf{P}$-a.s., one can thus adjust $\varepsilon$ so that the relevant estimates hold for large $n$.

Applying the conclusions of the previous paragraph, the proof of Proposition \ref{compactcase} can be followed exactly to yield the result at (\ref{vsrwresult}). Moreover, since we have local time equicontinuity and the distributional convergence of time-change measures given by Lemma \ref{nunconv3}, we also obtain the convergence of local times as at (\ref{ltconv1}), and the convergence of the CSRW $X^{n,\omega,\nu}$ under the annealed measure (cf.\ the proof of Proposition \ref{lbmconv} again).
\end{proof}

{\example\label{exa6-5} {\rm As a first application of Proposition \ref{rcmtreeresult}, one might consider the random conductance model on the Vicsek set example of Example \ref{lbmexamples}(ii). For this, we obtain the quenched convergence of the VSRW,
\[\left(X^{n,\omega}_{\varrho 15^nt}\right)_{t\geq 0}\rightarrow \left(X_t\right)_{t\geq0},\]
where $X$ is the Brownian motion on the Vicsek set, and also the annealed convergence of the CSRW,
\[\left(X^{n,\omega,\nu}_{2\varrho5^{n/\alpha}3^n t}\right)_{t\geq 0}\rightarrow\left(X^{\nu}_t\right)_{t\geq 0},\]
where $X^{\nu}$ is the $\alpha$-FIN diffusion on the Vicsek set.}}

\subsection{Random conductance model on self-similar fractals}\label{rcmfractalsec}

In this section, we study the random conductance model on a class of self-similar fractals, extending the homogenisation results of
\cite{Kum2,kk} greatly. After introducing the model in Section \ref{ufrsec}, we then go on to study the renormalisation and homogenisation of associated discrete Dirichlet forms in Sections \ref{renormsec} and \ref{homogsec}, respectively, and derive our main scaling results in Section \ref{concsec}.

\subsubsection{Uniform finitely ramified graphs}\label{ufrsec}

For $\beta>1$ and $I=\{1,2,\cdots, N\}$, let $(\Psi_i)_{i\in I}$ be a family of contraction maps on ${\br}^d$ such that $\Psi_i{\bf x}=\beta^{-1}U_i{\bf x} +\gamma_i,~{\bf x}\in {\br}^d$, where $U_i$ is a unitary map and $\gamma_i\in {\br}^d$. Assume that $(\Psi_i)_{i\in I}$ satisfies the open set condition, i.e., there is a non-empty, bounded open set $W$ such that $(\Psi_i (W))_{i\in I}$ are disjoint and $\cup_{i\in I} \Psi_i (W)\subset W$.
As $(\Psi_i)_{i\in I}$ is a family of contraction maps,
there exists a unique non-void compact set $F$ such that ${F}
=\cup_{i\in I}\Psi_i ({F})$. We assume $F$ is connected.

Let $Fix$ be the set of fixed points of the maps $\Psi_{i}$, $i\in I$. A point $x \in Fix$ is called an {\it essential fixed point} if there exist $i,j \in I,~i \ne j$ and $y\in Fix$ such that $\Psi_{i}(x)=\Psi_{j}(y)$. Let $I_{Fix}:=\{i\in I:$ the fixed point of $\Psi_i$ is an essential fixed point$\}$. We write ${V}_0$ for the set of essential fixed points. Denote $\Psi_{i_{1},\dots,i_{n}}=\Psi_{i_{1}}\circ\dots\circ\Psi_{i_{n}}$. We further assume a finitely ramified property, i.e., if $\{i_{1},\dots,i_{n}\}, \{j_{1},\dots,j_{n}\}$ are distinct sequences, then
\[\Psi_{i_{1},\dots,i_{n}}({ F}) \bigcap\Psi_{j_{1},\dots,j_{n}}
({F})=\Psi_{i_{1},\dots,i_{n}}({V}_0)\bigcap\Psi_{j_{1},
\dots,j_{n}}({ V}_0);\]
note that, for each $n\ge 0$ and $i_1,\cdots,i_{n}\in I$, we call a set of the form $\Psi_{i_1,\cdots, i_n}({ V}_0)$ an $n$-cell. A compact \emph{uniform finitely ramified (u.f.r.)\ fractal} ${F}$ is a set determined by $(\Psi_i)_{i\in I}$ satisfying  the above assumptions with $|{V}_0|\ge 2$. Throughout, we assume without loss of generality that $\Psi_1({\bf x})=\beta^{-1}{\bf x}$ and ${\bf 0}$ belongs to ${V}_0$. We observe that u.f.r.\ fractals, first introduced in \cite{HK}, form a class of fractals which is wider than nested fractals (\cite{lind}), and is included in the class of p.c.f. self-similar sets (\cite{kig1}). In particular, the Sierpi\'nski gasket is an example of a u.f.r.\ fractal.

We next introduce the sequence of u.f.r.\ graphs approximating $F$. In particular, let
\[{V}_n=\cup_{i_1,\cdots,i_n\in I}\Psi_{i_{1},\dots,i_{n}}({V}_0),\]
noting that ${F}$ is the closure of $\cup_{n=0}^\infty{V}_n$. Moreover, denote by $E_n$ the collection of pairs of distinct points $x,y\in V_n$ such that $x$ and $y$ are in the same $n$-cell, and let $\mu_n$ be the counting measure on $V_n$ (placing mass one on each vertex). We will be interested in the scaling behaviour of $(V_n,R_n^\omega,\mu_n,\rho_n)$ (where $\rho_n$ is some distinguished vertex) and the associated VSRW and CSRW when $R_n^\omega$ is the resistance metric determined by placing random conductances along edges in $E_n$; in this section we generalise slightly from the i.i.d.\ conductance assumption to allow dependencies within the same $n$-cell.

For some of our results, it will be convenient to work in terms of the unbounded u.f.r.\ fractal and graphs; we define these now. We call ${\hat F}:=\cup_{n=1}^{\infty} \beta^n {F}$ an unbounded uniform finitely ramified fractal. Let ${\hat V}={\hat V}_0=\cup_{n=0}^{\infty}\beta^{n} {V}_n$, and ${\hat V}_{n}=\beta^{-n} {\hat V}$ for $n \in \bz$. We define $n$-cells for $n\in \bz$ as in the compact case, and denote by ${\hat E}_n$ the edges of the unbounded graph, connecting vertices within the same $n$-cell.

Finally for this section, we introduce some useful index spaces. In particular, let
\begin{eqnarray*}
\Xi &=&\{\eta \in I^{\bz}: \mbox { there exists $n\in \bz$ such that }
\eta_k=1, k \ge n\},\\
\Xi_+ &=&\{\eta \in I^{\bn}: \mbox { there exists $n\in \bn$ such that }
\eta_k=1, k \ge n\}.\end{eqnarray*}
There is then a continuous map $\pi: \Xi\to \br^D$ such that
$$\pi(\eta)= \lim_{n\to \infty} \beta^n
\Psi_{\eta_n}(\Psi_{\eta_{n-1}}(\cdots(\Psi_{\eta_{-n}}
({\bf 0}))\cdots)).$$
It is easy to see ${\hat F}=\pi (\Xi)$.
For any $\eta\in \Xi
_+$ and $i\in I_{Fix}$, define $[\eta,i]\in \Xi$ as follows;
\[[\eta,i](k)=\left\{\begin{array}{rl}\eta_k,&~~k\ge 1\\i,&~~k\le
0.\end{array}\right.\]
Then, ${\hat V}=\{\pi([\eta,i]): \eta \in \Xi_+, i\in I_{Fix}\}.$

\subsubsection{Renormalisation of forms}\label{renormsec}

In this section, we introduce notation and basic properties for Dirichlet forms and associated renormalisation maps on uniform finitely ramified graphs. To begin with, let $\cq$ be the set of $Q=(Q_{ij})_{i,j\in I_{Fix}}$
such that
\[Q_{ij}=Q_{ji},~~\forall i,j\in I_{Fix},\qquad\sum_{j\in I_{Fix}}
Q_{ij}=0, ~~\forall i\in I_{Fix}.\]
Observe that $\cq$ is a vector space, with an inner product
$(\cdot,\cdot)_{\cq}$ given by
\[(Q,Q')_{\cq}=\sum_{j,k\in I_{Fix}} Q_{jk}Q'_{jk}=\mbox {Trace}~ Q{}^tQ',\qquad
Q,Q'\in \cq.\]
Let $\cq_+=\{Q\in \cq: {S}_Q(\xi,\xi)\ge 0$ for any $\xi \in l(I_{Fix})\}$, where
\[{S}_Q(\xi,\xi)=-\sum_{i,j\in I_{Fix}}Q_{ij}\xi_i \xi_j=\frac 12 \sum_{i,j\in I_{Fix}}Q_{ij}(\xi_i-\xi_j)^2,\]
and we define $l(A)=\{f:A\to \br\}$ for a set $A$. Set
\[\|Q\|^2=\sup_{\xi\in l(I_{Fix})}\frac{{S}_Q(\xi,\xi)}{\sum_{i\in I_{Fix}}\xi_i^2}.\]
Note that $c_{1}\|Q\|^2\le (Q,Q)_{\cq}\le c_{2}\|Q\|^2$ for all $Q\in \cq_+$. Let
\begin{eqnarray*}
\cq_M&:=&\{Q\in \cq: Q_{ij}\ge 0, \forall i,j\in I_{Fix}, i \ne j\},\\
\mbox{Int} (\cq_M)&:=&\{Q\in \cq: Q_{ij}> 0, \forall i,j\in I_{Fix}, i \ne j\},\\
\cq_{irr}&:=&\{Q\in \cq_M: {S}_Q(\xi,\xi)=0 \Leftrightarrow \xi \mbox{ is
constant}\}.
\end{eqnarray*}
Note that $\mbox{Int} (\cq_M)\subset \cq_{irr} \subset \cq_M \subset \cq_+$. Take $Q_*\in\mbox{Int} (\cq_M)$, and let
\begin{eqnarray*}
\cx_+:=C(\Xi_+,\cq_+),~~
\cx_M:=C(\Xi_+,\cq_M),~~\cx_{irr}:= C(\Xi_+,\cq_{irr}).
\end{eqnarray*}
Then $\cx_+$ and $\cx_M$ are convex cones. For any $\theta\in \cx_+$, define ${\hat S}_\theta$ by
\[{\hat S}_\theta(u,u)=\frac 12\sum_{\eta\in\Xi_+}{S}_{\theta(\eta)}(u(\pi([\eta,\cdot])), u(\pi([\eta,\cdot]))),~~u\in \Lte,\]
where ${\hat \mu}_0$ is the counting measure on ${\hat V}$. If $\theta\in \cx_M$, then ${\hat S}_\theta$ is a Dirichlet form on $\Lte$. So, there is an associated Markov process $((X^\theta_t)_{t\geq 0}, (P^\theta_x)_{x\in \hat{V}})$. We introduce an order relation $\le$ in $\cx_+$ as follows:
\[\theta\le {\theta'} ~\mbox { if }~ {\hat S}_\theta(u,u) \le {\hat S}_{\theta'} (u,u) ~\mbox { for all }~ u \in \Lte.\]
The norm on $\cx_+$ is given by $\|\theta\|^2=\sup_{u\in \Lte}{\hat S}_\theta(u,u)/\|u\|_{\Lte}^2$.

We now define the renormalisation map $\bar{\Phi}$.
 For any $\theta\in \cx_+$, let ${\hat S}_\theta^{(1)}:\Lte\to [0,\infty)$ be given by
\[{\hat S}_\theta^{(1)}(u)=\inf \{{\hat S}_\theta(v,v):v\in \Lte,~
v(\beta x)=u(x),~x\in {\hat V}\}.\]
By the self-similarity of ${\hat F}$, there is a renormalisation map ${\bar \Phi}:\cx_+ \to \cx_+$ defined by setting ${\hat S}_\theta^{(1)}(u)={\hat S}_{{\bar \Phi}(\theta)} (u,u)$ for all $\theta\in \cx_+$ and $u \in \Lte$. Let $\iota: \cq_+\to \cx_+$ be such that $\iota (Q)(\eta)=Q$ for all $\eta\in\Xi_+$ and $Q\in \cq_+$. Define a renormalisation map ${\tilde \Phi}:\cq_+\to\cq_+$ as ${\tilde \Phi}(Q)={\bar \Phi}(\iota(Q))(\eta)$ for $\eta\in \Xi_+$. Note that it is independent of the choice of $\eta\in\Xi_+$. By Schauder's fixed point theorem, we know that there exists $Q_*\in \cq_M$ (with $(Q_*)_{ij}>0$ for some $i\ne j$) and $\varrho_{Q_*}>0$ such that ${\tilde \Phi}(Q_*)=\varrho_{Q_*}^{-1} Q_*$. Henceforth, we assume the following.

\begin{assu}\label{thm:ass1}
(1) For each $Q\in \cq_{irr}$, there exists $n_0=n_0(Q)\in \bn$ such that ${\tilde \Phi}^n(Q)\in \mbox{Int} (\cq_M)$ for all $n\ge n_0$.\\
(2) There exists $Q_0\in \mbox{Int} (\cq_M)$ and $\varrho_{Q_0}>0$ such that ${\tilde \Phi}(Q_0)=\varrho_{Q_0}^{-1} Q_0$.
\end{assu}

In the following, we take one $Q_0\in \mbox{Int} (\cq_M)$, as given by Assumption \ref{thm:ass1}(2), and fix it.

\begin{rem}{\rm
(1) It is known that $\varrho_{Q_0}>0$ is uniquely determined, i.e.\ if $Q_1,Q_2\in \cq_{irr}$ satisfy ${\tilde \Phi}(Q_j)=\varrho_{Q_j}^{-1}  Q_j$  ($j=1,2$) with $\varrho_{Q_1}, \varrho_{Q_2}>0$, then $\varrho_{Q_1}=\varrho_{Q_2}=\varrho_{Q_0}$. In the class of fractal graphs we consider, we can prove $\varrho_{Q_0}>1$ (see \cite{kig1}, for example).\\
(2) Every nested fractal satisfies Assumption \ref{thm:ass1}.}
\end{rem}

Under Assumption \ref{thm:ass1}, we set $\Phi=\varrho_{Q_0} {\bar \Phi}: \cx_+\to\cx_+$, $\hat\Phi=\varrho_{Q_0} {\tilde \Phi}$ and ${\hat S}_\theta^{\Phi}(u)=\varrho_{Q_0}{\hat S}_\theta^{\bar \Phi}(u)$ for $u \in \Lte$.

\subsubsection{Homogenisation of forms}\label{homogsec}

In this subsection, we will describe the homogenisation of the discrete Dirichlet forms associated with the random conductance model on u.f.r.\ fractals, see Theorem \ref{thm:conth} for the main result. First, we give some further definitions for later use. Let ${V}_0=\{a_i: i\in I_{Fix}\}$. For $Q^*\in \mbox{Int} (\cq_M)$, $k\in I$, we define a matrix $A_{k,Q^*}\in l(I_{Fix}^2)$ by setting
\[(A_{k,Q^*})_{ij}=P^{Q^*}_{\Psi_k(a_i)}\big(X^1_{\tau_{{ V}_0}}=a_j\big),\]
where $X^1$ is a discrete time Markov chain on ${V}_1$ whose transition probabilities are determined by the Dirichlet form obtained by placing a copy of $Q^*$ on each $1$-cell, and $\tau_{{V}_0}=\inf\{n\ge 0: X^1_n\in {V}_0\}$. Then, it is easy to see that the following holds for u.f.r.\ graphs; $0<(A_{k,Q^*})_{ij}<1$ if $k\ne i$ and $(A_{k,Q^*})_{kj}=\delta_{kj}$.

We now define a liberalisation of the renormalisation map around the fixed point $Q_*$.
For any $\theta\in \cx_+$ and $Q_*\in \mbox{Int} (\cq_M)$ with ${\hat \Phi}(Q_*)=Q_*$, define
\[{\hat S}_\theta^{(2)}(u)=\varrho_{Q_0}{\hat S}_\theta(v,v)\qquad\mbox { for }~~~u \in \Lte,\]
where $v\in \Lte$ satisfies $v(\beta x)=u(x),~x\in {\hat V}$, and $v$ is $Q_*$-harmonic on ${\hat V}\setminus \beta {\hat V}$, i.e.,
\[v(\pi([\eta\cdot i,j]))=
\sum_{k\in I_{Fix}} (A_{i,Q_*})_{jk}u(\pi([\eta,k]))\qquad\mbox { for }
~~~i\in I,j\in I_{Fix}.\]
Here $\eta\cdot i\in \Xi_+$ is given by $(\eta\cdot i)_n=\eta_{n-1},~n\ge2$  and $(\eta\cdot i)_1=i$.
 It is easy to see that ${\hat S}_\theta^{(2)}(u)={\hat S}_{{H_{Q_*}}(\theta)}(u,u)$ for all $\theta\in \cx_+$ and $u \in \Lte$, where
\[
{H_{Q_*}}(\theta)(\eta)=\varrho_{Q_0}\sum_{k\in I}
{}^tA_{k,Q_*} \theta(\eta\cdot k)A_{k,Q_*}.
\]
Similarly, we define a linear map $\hat H_{Q_*}: \cq_+\to \cq_+$ by
${\hat H_{Q_*}}(Q)=\varrho_{Q_0}\sum_{k\in I}
{}^tA_{k,Q_*} QA_{k,Q_*}.$
Note that $\hat H_{Q_*}(Q_*)=Q_*$. The following properties of $\Phi$ and $H_{Q_*}$ are easy, but important. Note that the corresponding results hold for $\hf$.

\begin{lemma} Let $Q_*\in \mbox{Int} (\cq_M)$
satisfy ${\hat \Phi}(Q_*)=Q_*$ and $\theta,{\theta'}\in \cx_+$.\\
(1) If $\theta\le {\theta'}$, then $\Phi(\theta)\le \Phi({\theta'})$, $H_{Q_*}(\theta)\le
H_{Q_*}({\theta'})$ and $\Phi(\theta)\le H_{Q_*}(\theta)$.\\
(2) For $a,b\ge 0$, $\Phi(a\theta+b{\theta'})\ge a\Phi(\theta)+b\Phi({\theta'})$ and
$H_{Q_*}(a\theta+b{\theta'})= aH_{Q_*}(\theta)+bH_{Q_*}({\theta'})$.
\end{lemma}

We are now ready to introduce a probability measure $\mathbf{P}$ on $\Theta_M$ to describe our random conductance model in this setting. In particular, we now write $\theta$ for a $\Theta_M$-valued random variable, and suppose that, under $\mathbf{P}$, the elements $(\theta(\eta))_{\eta\in \Xi_+}$ are independently identically distributed
$\cq_M$-valued random variables such that
$C_1 Q_0\le \theta(\eta)$ for $\eta\in \Xi_+$. Note that in \cite{Kum2,kk} it was assumed that $\mathbf{P}(\{\theta\in \cx_M:  C_1 Q_0\le \theta(\eta) \le C_2 Q_0,$ for $\eta\in \Xi_+\})=1$ for some $C_1,C_2>0$. Here we do not assume
such a uniform ellipticity condition from above. We note the following further property of $\Phi$:
\begin{equation}\label{basic10}
\mathbf{E}(\Phi(\theta))\le \Phi(\mathbf{E}(\theta)),
\end{equation}
where the expectation is taken for each element of the matrix in $\cq_M$.

Let $\Phi^n$ be the $n$-th iteration of $\Phi$. We make the following further assumption, which is possible to verify in the case of nested fractals when the distribution of the individual conductances does not have too heavy a tail at infinity.

\begin{assu}\label{thm:ass2} There exists $n_0\in \bn$ such that
\[\mathbf{E}\left[(\Phi^{n_0}(\theta)(\eta)_{ij})^2\right]<\infty,\qquad \forall i, j\in I_{Fix},\:\eta\in\Xi_+.\]
\end{assu}

Note that under Assumption $\ref{thm:ass2}$ we have, for all $i\ne j\in I_{Fix}$, $\eta\in\Xi_+$,
\begin{equation*}
\mathbf{E}[\Phi^{n_0}(\theta)(\eta)_{ij}]
\le (\mathbf{E}[(\Phi^{n_0}(\theta)(\eta)_{ij})^2])^{1/2}<\infty,
\end{equation*}
so by (\ref{basic10}), $\mathbf{E}[\Phi^{n}(\theta)(\eta)_{ij}]<\infty$ for all $i\ne j\in I_{Fix}$, $\eta\in\Xi_+$, $n\ge n_0$. We next give a sufficient condition for Assumption $\ref{thm:ass2}$ to hold. For $x,y\in V_n$, define
\[h_n(x,y)=\min\{k: K_1,\cdots K_k \mbox{ are $n$-cells}, x\in K_1,y\in K_k, K_i\cap K_{i+1}\ne \emptyset, \forall i=1,\dots,k-1\}.\]

\begin{propn} Suppose that there exists $n_0\in \bn$ such that $\min_{x,y\in V_0, x\ne y}h_{n_0}(x,y)\ge 2$. Suppose also that the law of
$\theta(\eta)_{ij}$ has at most polynomial decay at infinity
for all $i\ne j\in I_{Fix}$, $\eta\in\Xi_+$, namely there exists $c_1,\gamma_{ij}>0$ such that $\mathbf{P} (\theta(\eta)_{ij}\ge s)\le c_1s^{-\gamma_{ij}}$. Then Assumption~\ref{thm:ass2} holds. In particular, Assumption~\ref{thm:ass2} holds for nested fractal graphs if the law of the random conductances has at most polynomial decay at infinity.
\end{propn}

\begin{proof} First, suppose we have two edges with conductance $\omega_1$, $\omega_2$ such that $\mathbf{P}(\omega_i\ge s)\le c_is^{-\gamma_i}$ for $i=1,2$. If the edges are connected in parallel, then the effective conductance is $\omega_1+\omega_2$, which satisfies
\begin{equation}\label{eq:biebi11}
\mathbf{P}(\omega_1+\omega_2\ge s)\le \mathbf{P}(\omega_1\ge s/2)+\mathbf{P}(\omega_2\ge s/2)\le 2(c_1\vee c_2)s^{-\gamma_1\wedge \gamma_2},\qquad\forall s\ge 1.
\end{equation}
Similarly, connect the two conductances in series, and assume that $\omega_1$ and $\omega_2$ are independent. Then the effective conductance is $(\omega_1^{-1}+\omega_2^{-1})^{-1}$, and we have
\begin{eqnarray}
\mathbf{P}((\omega_1^{-1}+\omega_2^{-1})^{-1}\ge s)
&=&\mathbf{P}(\omega_1^{-1}+\omega_2^{-1}\le s^{-1})\nonumber\\
&\le &\mathbf{P}(\omega_1^{-1}\le s^{-1})\mathbf{P}(\omega_2^{-1}\le s^{-1})\nonumber\\
&\le &c_1c_2s^{-\gamma_1- \gamma_2},
\qquad\forall s\ge 1.\label{eq:biebi22}
\end{eqnarray}
Next, note that by the assumption we have $\min_{x,y\in V_0, x\ne y}h_{ln_0}(x,y)\ge 2^l$
for all $l\ge 1$. Let $a_i\in  V_0$ be the fixed point of $\Psi_i$. Consider the network
on $\beta^{ln_0}V_{ln_0}$ and fix $a_i\ne a_j\in V_0$. Define $H_m=\{z\in \beta^{ln_0}V_{ln_0}: h_{ln_0}(a_i,\beta^{-ln_0}z)=m\}$ for $1\le m \le h_{ln_0}(a_i,a_j)-1$, and $H_{h_{ln_0}(a_i,a_j)}=\{z\in \beta^{ln_0}V_{ln_0}: h_{ln_0}(a_i,\beta^{-ln_0}z)\ge h_{ln_0}(a_i,a_j)\}$. Now short all the vertices that are in the same $H_m$ for $1\le m \le h_{ln_0}(a_i,a_j)$, and let $C_{ij}$ be the effective conductance between $a_i$ and $a_j$ for the induced network. By Rayleigh's monotonicity principle for electric networks, we see that $\Phi^{ln_0}(\theta)(\eta)_{ij}\le C_{ij}$, where $\eta=(1,1,1,\dots)$. Applying \eqref{eq:biebi11} and \eqref{eq:biebi22} repeatedly, we see that $C_{ij}^2$ is integrable when $l$ is large enough.  Therefore Assumption $\ref{thm:ass2}$ holds in this case. Finally, note that the condition $\min_{x,y\in \hat V_0, x\ne y}h_{n_0}(x,y)\ge 2$ holds for nested fractal graphs due to \cite[Lemma $(2.8)$]{kus}, so the last assertion holds.
\end{proof}

We are now ready to state the main result of this section.

\begin{thm}\label{thm:conth} Under Assumptions $\ref{thm:ass1}$ and $\ref{thm:ass2}$, there exists $Q_{\mathbf{P}}
\in\mbox{Int}(\cq_{M})$ such that, for all $\eta\in \Xi_+$,
\begin{equation}\label{eq:daiji}
Q_{\mathbf{P}}=\lim_{n\to\infty}\Phi^n(\theta)(\eta),\qquad
\mbox{in }{L}^1(\cq_M,\mathbf{P}).
\end{equation}\end{thm}

The rest of this subsection is devoted to proving Theorem \ref{thm:conth}. The next proposition is a restricted version of the result by Peirone \cite{peir}, whose original ideas come from Sabot \cite{sab}; see \cite[Appendix $A$]{Kum2} for the proof.

\begin{propn}\label{thm:peir}
Under Assumption $\ref{thm:ass1}$, for each $M\in \cq_{irr}$,
there exists $Q_M \in \mbox{Int}(\cq_M)$ such that
$Q_M=\lim_{n\to\infty} \hf^n(M)$.
\end{propn}

The next lemma is an adaptation of \cite[Lemma 4.1]{kk}, but the proof requires serious modification from the latter work to cover our more general setting. We denote by $H^n_{Q_*}$ the $n$-th iteration of $H_{Q_*}$.

\begin{lemma}\label{thm:EHX}
Let $Q_*\in \mbox{Int}(\cq_M)$ satisfy $\hf (Q_*)=Q_*$. Under Assumption \ref{thm:ass2}, there exist $c_{1}>0$ and $0<\varepsilon <1$ such that
\begin{equation}\label{eq:neow}
\mathbf{E}[\Vert H^n_{Q_*}(\Phi^{n_0}(\theta))(\eta)-H^n_{Q_*}(\mathbf{E}[\Phi^{n_0}(\theta)])(\eta)\Vert^2]\le c_{1} (1-\varepsilon)^n,~~~\forall\eta\in \Xi_+, n\ge 1.
\end{equation}
In particular, it $\mathbf{P}$-a.s.\ holds that
\[\lim_{n\to\infty}\Vert H^n_{Q_*}(\Phi^{n_0}(\theta))(\eta)-H^n_{Q_*}(\mathbf{E}[\Phi^{n_0}(\theta)])(\eta)\Vert= 0,\qquad\forall\eta\in \Xi_+.\]
\end{lemma}
\begin{proof}
Let the left hand side of $(\ref{eq:neow})$ be $f(n,\eta)$ and set $\theta'=\Phi^{n_0}(\theta)$.
Further, let
\begin{eqnarray*}
\theta'^{(1)}_{i_1,\cdots,i_n}(\eta)&=&{}^tA_{i_n}\cdots{}^tA_{i_1}
\theta'(\eta\cdot i_1\cdots\cdot i_n)A_{i_1}\cdots A_{i_n},\\
\theta'^{(2)}_{i_1,\cdots,i_n}(\eta)&=&{}^tA_{i_n}\cdots{}^tA_{i_1}
\mathbf{E}[\theta'(\eta\cdot i_1\cdots\cdot i_n)]A_{i_1}\cdots A_{i_n},\\
\theta'_{i_1,\cdots,i_n}(\eta)&=&\theta'^{(1)}_{i_1,\cdots,i_n}(\eta)
-\theta'^{(2)}_{i_1,\cdots,i_n}(\eta),
\end{eqnarray*}
where we set $A_i:=A_{i,Q_*}$. Then we have
\begin{eqnarray*}
f(n,\eta) &\leq& c\varrho_{Q_0}^{2n}
\mathbf{E}\left[\mbox{Trace}\left[\left(\sum_{i_1,\cdots,i_n}\theta'_{i_1,\cdots,i_n}(\eta)\right)^2\right]\right]\\
&=&c
\varrho_{Q_0}^{2n}\sum_{i_1,\cdots,i_n}
\mathbf{E}\left[\mbox{Trace}\left[\left(\theta'_{i_1,\cdots,i_n}(\eta)\right)^2\right]\right]\\
&= &c\varrho_{Q_0}^{2n}\sum_{i_1,\cdots,i_n}
\left(\mathbf{E}\left[\mbox{Trace}\left[\left(\theta'^{(1)}_{i_1,\cdots,i_n}(\eta)\right)^2\right]\right]-\mbox{Trace}\left[\left(\theta'^{(2)}_{i_1,\cdots,i_n}(\eta)\right)^2\right]\right)\\
&\le& c\varrho_{Q_0}^{2n}
\sum_{i_1,\cdots,i_n}\mathbf{E}\left[\left(\mbox{Trace}\:\theta'^{(1)}_{i_1,\cdots,i_n}(\eta)\right)^2\right],
\end{eqnarray*}
where the first equality is because $\theta'_{i_1,\cdots,i_n}(\eta)$ and $\theta'_{j_1,\cdots,j_n}(\eta)$ are independent (because of the finitely ramified property) and mean zero for $({i_1,\cdots,i_n})\ne (j_1,\cdots,j_n)$, and the last inequality is because
$\mbox{Trace } (B^2) \le (\mbox{Trace } B)^2$ for any non-negative definite
symmetric matrix $B$.

Set $A=(a_{ij})=A_{i_1}\cdots A_{i_n}$, $(x_{ij})
=\theta'(\eta\cdot i_1\cdots\cdot i_n)$.
Then we have
\begin{eqnarray*}
(\theta'^{(1)}_{i_1,\cdots,i_n}(\eta))_{ij}=\sum_{k,l}a_{ki}a_{lj}x_{kl}
=-\frac 12\sum_{k,l: k\ne l}(a_{ki}-a_{li})(a_{kj}-a_{lj})x_{kl},
\end{eqnarray*}
because $x_{ll}=-\sum_{k:k\ne l}x_{kl}$. Thus, denoting
$Q_*=(q_*)_{ij}$, we have
\begin{eqnarray*}
\lefteqn{\mathbf{E}\left[\left(\mbox{Trace}\:\theta'^{(1)}_{i_1,\cdots,i_n}(\eta)\right)^2\right]}\\
&=&\mathbf{E}\left[\left(-\frac 12\sum_i\sum_{k,l: k\ne l}(a_{ki}-a_{li})^2x_{kl}\right)^2\right]\\
&=&\frac 14\sum_i\sum_{k,l: k\ne l}\sum_{i'}\sum_{k',l': k'\ne l'}
(a_{ki}-a_{li})^2(a_{k'i'}-a_{l'i'})^2\mathbf{E}[x_{kl}x_{k'l'}]\\
&\le&\frac 14\sum_i\sum_{k,l: k\ne l}\sum_{i'}\sum_{k',l': k'\ne l'}
(a_{ki}-a_{li})^2(a_{k'i'}-a_{l'i'})^2(\mathbf{E}[x_{kl}^2])^{1/2}(\mathbf{E}[x_{k'l'}^2])^{1/2}\\
&=&\frac 14\left(\sum_i\sum_{k,l: k\ne l}(a_{ki}-a_{li})^2(\mathbf{E}[x_{kl}^2])^{1/2}\right)^2\\
&\le& c_1\left(\sum_i\sum_{k,l: k\ne l}(a_{ki}-a_{li})^2
(q_*)_{kl}\right)^2,
\end{eqnarray*}
where the last inequality is because there exists $c_*>0$ such that $\mathbf{E}[(\Phi^{n_0}(\theta)(\eta)_{ij})^2]\le c_*$ for all $i, j\in I_{Fix}$, which is due to Assumption $\ref{thm:ass2}$. In particular, we obtain that
\begin{eqnarray*}
f(n,\eta) &\le&c_2\varrho_{Q_0}^{2n}\sum_{i_1,\cdots,i_n}
\{\mbox{Trace }{}^tA_{i_n}\cdots{}^tA_{i_1}
Q_* A_{i_1}\cdots A_{i_n}\}^2.
\end{eqnarray*}
Now, from the proof of \cite[Proposition (5.5)]{kus}, we have
\[\varrho_{Q_0}^{n}~ {}^tA_{i_n}\cdots{}^tA_{i_1}Q_* A_{i_1}\cdots  A_{i_n} \le(1-\varepsilon)^n Q_*\]
for some $0 <\varepsilon<1$. (Note that in \cite{kus} it is assumed that
$\sum_{k\ne i}{^tA_kA_k}$ is strictly positive for all $i$, but this assumption is satisfied in our setting; see \cite[Proposition $(7.2)$]{kus}.) Combining this with $\hat H_{Q_*}(Q_*)=Q_*$, we obtain
\[f(n,\eta) \le c'_*  (1-\varepsilon)^n (\mbox{Trace }Q_*)^2\le c_1 (1-\varepsilon)^n.\]
\end{proof}

\begin{proof}[Proof of Theorem \ref{thm:conth}] Let $\phi_m=\mathbf{E}[\Phi^{m+n_0}(\theta)(\eta)]$ ($\phi_m$ is independent of $\eta$). By Assumption $\ref{thm:ass2}$ and Proposition \ref{thm:peir}, for each $m\in \bn$, there exists $Q_m\in \mbox{Int} (\cq_M)$ such that $\lim_{n\to\infty} \hf^n(\phi_m)=Q_m$ and $\hf(Q_m)=Q_m$. On the other hand,
similarly to (\ref{basic10}) we see
\begin{equation}\label{eq:rnrm}
\hf^n (\phi_m)\ge \phi_{n+m}\qquad\forall m,n\in \bn\cup\{0\},
\end{equation}
so that $Q_m\ge Q_{n+m}$. Denote the limit of $(Q_m)_{m\geq 0}$ by $Q_+$; then $\hf(Q_+)=Q_+$. (Note that $Q_+\in \mbox{Int}(\cq_M)$ due to Assumption \ref{thm:ass1}(1) and the assumption $\mathbf{P}(\{\theta\in \cx_M:  C_1 Q_0\le \theta(\eta)$ for $\eta\in \Xi_+\})=1$.) For any $\varepsilon>0$, there exists $N_{\varepsilon}\in \bn$ such that
\begin{equation}\label{eq:12-add}
(1+\varepsilon)Q_+\ge \phi_m\qquad \forall m\ge N_{\varepsilon}.
\end{equation}
Indeed, if this does not hold, then because there exists $C_*>0$ such that $(\phi_m)_{ij}\le C_*$ for all $i\ne j\in I_{Fix}$ and all $m\in \bn$, there exists a subsequence $(l_j)_{j\geq 0}$ such that $\phi_{l_j}\ge (1+\varepsilon)Q_+$ and $\lim_{j\to\infty}\phi_{l_j}=:{\bar \phi}$ exists. On the other hand, by (\ref{eq:rnrm}), we have $\hf^{l_{j'}-l_j}(\phi_{l_j})\ge \phi_{l_{j'}}$ for all $j'\ge j$ so that $Q_+\ge {\bar \phi}$, which is a contradiction. By the definition of $Q_m$, for each $m$ and $\varepsilon>0$, there exists $L_{m,\varepsilon}$ such that $(1-\varepsilon)Q_m\le \hf^n (\phi_m)$ for all $n\ge L_{m,\varepsilon}$. Combining these facts and noting $\hh_{Q_+}^n(\phi_m)\ge \hf^n(\phi_m)$, we have
\begin{equation}\label{eq:12-1}
(1-\varepsilon)Q_+\le \hh_{Q_+}^n (\phi_m)\le (1+\varepsilon)Q_+\qquad \forall n\ge L_{m,\varepsilon}, m\ge N_{\varepsilon}.
\end{equation}
On the other hand, by Lemma \ref{thm:EHX}, we have $\mathbf{P}$-a.s.\ that
\[\lim_{n\to\infty}\|H_{Q_+}^n(\Phi^{m+n_0}(\theta))(\eta)-\hh_{Q_+}^n(\phi_m)\|=\lim_{n\to\infty}\|H_{Q_+}^n(\Phi^{m+n_0}(\theta))(\eta)- H_{Q_+}^n(\iota(\phi_m))(\eta)\|= 0\]
for all $\eta\in \Xi_+$, $m\geq0$. Since $H_{Q_+}^n(\Phi^{m+n_0}(\theta))(\eta)\ge \Phi^{n+m+n_0}(\theta)(\eta)$, we see that the following holds $\mathbf{P}$-a.s.: for some $N'_{\varepsilon,\eta}\in\bn$,
\begin{equation}\label{eq:12-3}
(1+\varepsilon)Q_+\ge  \Phi^{m+n_0}(\theta)(\eta),\qquad \forall \eta\in \Xi_+, m\ge N'_{\varepsilon,\eta}.
\end{equation}
We now establish some more properties of $\hh_{Q_+}$. It is easy to see $\sup_{n}|||\hh^n_{Q_+}|||<\infty$,
where $|||\hh^n_{Q_+}|||:=\sup_{Q\in \mathcal{Q}_M, \|Q\|=1}\|\hh^n_{Q_{+}}(Q)\|$, see \cite[Lemma 4.3]{kk}. Using this, we see that the size of each Jordan cell corresponding to the largest eigenvalue of $\hh_{Q_+}$ is $1$. We thus obtain that there exists an orthogonal projection $P_0: \cq_M\to\cq_M$ so that for each $k\in \bn$, there exists $n_k\in \bn$ such that
\begin{equation}\label{eq:12-4}
|||\hh^{n_k}_{Q_+}-P_0|||\le 2^{-k}.
\end{equation}
By (\ref{eq:12-1}) and (\ref{eq:12-4}), we have $\phi_m\ge P_0\phi_m\ge (1-\varepsilon)Q_+$ for all $m\ge N_{\varepsilon}$. Together with (\ref{eq:12-add}), we have
\begin{equation}\label{eq:12-5}
\lim_{n\to\infty}\phi_n=Q_+.
\end{equation}
Now, by Fatou's lemma and (\ref{eq:12-5}),
\begin{equation}\label{eq:fatoo}
\mathbf{E}\big[\limsup_{n\to\infty}S_{\Phi^n(\theta)(\eta)}(u,u)\big]
\ge \limsup_{n\to\infty}S_{\phi_n}(u,u)= S_{Q_+}(u,u),
\end{equation}
for all $\eta\in \Xi_+,~u\in l(V_{\eta})$, where $V_{\eta}:=\{\pi([\eta,i]): i\in I_{Fix}\}$ is a $0$-cell whose address is $\eta$. (Note that we can use Fatou's lemma thanks to (\ref{eq:12-3}).) By (\ref{eq:12-3}) and (\ref{eq:fatoo}), we have
\[\limsup_{n\to\infty}S_{\Phi^n(\theta)(\eta)}(u,u)= S_{Q_+}(u,u),\]
$\mathbf{P}$-a.s.\ for all $\eta\in \Xi_+,~u\in l(V_{\eta})$. Applying \cite[Lemma 4.2]{kk} with $Y_n=S_{\Phi^n(\theta)(\eta)}(u,u)$ and $Y=S_{Q_+}(u,u)$
(note that $\sup_n \mathbf{E}[Y_n^2]<\infty$ due to Assumption \ref{thm:ass2}),
we have
\[\lim_{n\to\infty}\mathbf{E}\big[|S_{\Phi^n(\theta)(\eta)}(u,u)
-S_{Q_+}(u,u)|\big]
= 0,\qquad\forall \eta\in \Xi_+,~u\in l(V_{\eta}).\]
Since $l(V_{\eta})$ is finite dimensional, we obtain (\ref{eq:daiji}) where $Q_{\mathbf{P}}=Q_+$.
\end{proof}

\subsubsection {Application to the random conductance model}\label{concsec}

We are now ready to explain the application of the homogenisation results of the previous section to the random conductance model; see Proposition \ref{ssfrcmresult} for the main result. For the setting, we recall the graphs $(V_n,E_n)$, and the associated counting measure $\mu_n$, from Section \ref{ufrsec}. We further suppose each graph is equipped with a collection of random conductances $(\omega^n_e)_{e\in E_n}$ such that the conductances within each $n$-cell, $(\omega^n_e)_{e\subseteq \Psi_{i_1,\dots,i_n}(V_0)}$, are independent, and identically distributed as $(\omega^0_e)_{e\in E_0}$ (and built on a probability space with probability measure $\mathbf{P}$). The associated random resistance metric will be denoted by $R_n^\omega$.

Note that this family of random graphs can be coupled with the framework of the previous section. In particular, suppose that $(\theta(\eta)_{ij})_{i,j=1}^{I_F}$ is distributed as $(\omega^0_{a_i,a_j})_{i,j=1}^{I_F}$, independently for each $\eta$. Then we easily see that the random weighted graph $(V_n,E_n,\omega^n)$ is identical in distribution to that given by the conductances associated with $\theta$ on $\beta^nV_n\subseteq\hat{V}$. We will fix this identification throughout the section, and typically suppose that Assumptions \ref{thm:ass1} and  \ref{thm:ass2} are satisfied accordingly. This means that we can define the $Q_\mathbf{P}$ for which the conclusion of Theorem \ref{thm:conth} holds.

We next describe the limiting object. First, let $R_n$ be the resistance metric on $V_n$ induced by placing conductances according to $Q_{\mathbf{P}}$ along edges of $n$-cells, i.e.\ setting the conductance from $\Psi_{i_1,\dots,i_n}(a_i)$ to $\Psi_{i_1,\dots,i_n}(a_j)$ to be  $(Q_\mathbf{P})_{ij}$. From the fact that $\hat{\Phi}(Q_{\mathbf{P}})=Q_{\mathbf{P}}$, it follows that there exists a resistance metric $R$ on $V_*:=\cup_{n\geq 0}V_n$ defined by setting $R:=\varrho_{Q_0}^{-n}R_n$ on $V_n$, where $\varrho_{Q_0}>1$ is the scaling factor given by Assumption \ref{thm:ass1}. Moreover, by \cite[Theorems 2.3.10 and 3.3.4]{kig1}, taking the completion of the metric space $(V_*,R)$ yields a resistance metric $R$ on the u.f.r.\ fractal $F$, which is topologically equivalent to the Euclidean metric. It is moreover an elementary exercise to check that $N^{-n}\mu_n\rightarrow \mu$, where $\mu$ is the (unique up to a constant multiple) self-similar measure on $F$, placing equal weight on each $1$-cell; this measure is non-atomic and has full-support. We observe that, for any $\rho_n\in V_n$ such that $\rho_n\rightarrow \rho$ (with respect to $R$, or equivalently the Euclidean metric), we have that $(V_n,a_nR_n,b_n\mu_n,\rho_n)\rightarrow(F,R,\mu,\rho)$ in $\mathbb{F}_c$ with respect to the Gromov-Hausdorff-vague topology for $a_n=\varrho_{Q_0}^{-n}$, $b_n=N^{-n}$. We moreover note that $(V_n,a_nR_n,b_n\mu_n,\rho_n)_{n\geq 1}$ satisfies
UVD (see \cite[Lemma 3.2]{HK}).

As in Section \ref{rcmtreesec}, to get from the convergence of the previous paragraph to the convergence of the VSRW associated with the random conductances $(\omega^n_e)_{e\in E_n}$, we need to establish the convergence of the random metric $R_n^\omega$. This is the aim of the following lemma.

{\lemma \label{ssfresconv} Suppose Assumptions \ref{thm:ass1} and  \ref{thm:ass2} hold, and that the conductances $(\omega_e^0)_{e\in E_0}$ are uniformly bounded from below (i.e.\ there exists a constant $c>0$ such that $\omega_e^0\geq c$, $\mathbf{P}$-a.s.). Then it is the case that in $\mathbf{P}$-probability
\[\sup_{x,y\in V_n}a_n\left|R_n^\omega(x,y)-R_n(x,y)\right|\rightarrow 0.\]}
\begin{proof} Translating Theorem \ref{thm:conth} into the present notation, and noting that, for a finite network, convergence of edge conductances implies convergence of the resistance metric (cf.\ the proof of Lemma \ref{l2}), we obtain for any $x,y\in V_*$ that,
in $\mathbf{P}$-probability, $a_nR_n^\omega(x,y)\rightarrow R(x,y)$. Moreover, the fact that conductances are uniformly bounded below implies that there exists a constant $c_1$ such that $R_n^\omega\leq c_1R_n$, $\mathbf{P}$-a.s. From these two facts, one can deduce the result by following the argument of Lemma \ref{resconv}.
\end{proof}

To prove convergence of the CSRW to the $\alpha$-FIN diffusion, we introduce the random time-change measures $\nu_n$, as given by $\nu_n(\{x\})=\sum_{e\in E_n:\:x\in e}\omega_e^n$. In Lemma \ref{nunconvssf}, we will prove convergence to the limiting FIN measure $\nu$, again obtained from a Poisson process on $(0,\infty)\times F$ with intensity $\alpha v^{-1-\alpha}dv\mu(dx)$, under the following assumption. We note, in this setting, it makes sense to state convergence results with respect to the original Euclidean topology, since the objects already have a natural (non-isometric) embedding there. Moreover, we observe that the assumption is satisfied for i.i.d.\ edge weights, each with tails satisfying the same distributional asymptotics.

{\assu\label{ssfassu} There exists a constant $c>0$ such that the random conductance distribution satisfies
\[\mathbf{P}\left(\sum_{e\in E_0}\omega^0_e>u\right)
\sim c u^{-\alpha}\]
for some $\alpha\in(0,1)$.}

{\lemma\label{nunconvssf} Suppose Assumption \ref{ssfassu} holds, then there exists a constant $c_0>0$ such that
$c_0^{-1}b_n^{1/\alpha}\nu_n\rightarrow \nu$ in distribution with respect to the weak topology for finite measures on
$\mathbb{R}^d$.}

\begin{proof} The proof is again similar to the tree case (Lemma \ref{nunconv3}). In particular, it is an easy exercise to check that there exists a constant $c_0>0$ such that, for any $i_1,\dots,i_m\in\{1,\dots,N\}$, $b_n^{1/\alpha}\nu_n(\Psi_{i_1,\dots,i_m}(F))\rightarrow c_0\nu(\Psi_{i_1,\dots,i_m}(F))$ in distribution. From this, the result again follows from \cite[Theorem 16.16]{Kall}.
\end{proof}

From Lemmas \ref{ssfresconv} and \ref{nunconvssf}, we are able to prove the main result of this section by a similar argument to the proof of Proposition \ref{rcmtreeresult}; we thus state it without proof. We write $\mathbb{P}^{{\rm VSRW}_n}_x$ for the annealed law of the VSRW $X^{n,\omega}$ on the graph $V_n$ with conductances $\omega^n$, started from $x$. (Note that in Proposition \ref{rcmtreeresult} convergence of VSRW is shown $\mathbf{P}$-a.s., but here we have only annealed convergence since the convergence in Theorem \ref{thm:conth} is in the
${L}^1$-sense.) We write $\mathbb{P}^{{\rm CSRW}_n}_{x}$ for the annealed law of the corresponding CSRW $X^{n,\omega,\nu}$. We write $P_x$ for the law of the Brownian motion on $(F,R,\mu)$ started from $x$, and $\mathbb{P}^{{\rm FIN}}_{\rho}$ is the annealed law of the associated $\alpha$-FIN diffusion, defined as in Section \ref{bouchsec}.

{\propn\label{ssfrcmresult} Suppose Assumptions \ref{thm:ass1} and  \ref{thm:ass2} hold, and that the conductances $(\omega_e^0)_{e\in E_0}$ are uniformly bounded from below. It is then the case that
\[\mathbb{P}^{{\rm VSRW}_n}_{\rho_n}\left(\left(X^{n,\omega}_{t/ a_nb_n}\right)_{t\geq 0}\in\cdot\right)\rightarrow {P}_\rho\left(\left(X_{t}\right)_{t\geq 0}\in\cdot\right)\]
weakly as probability measures on $D(\mathbb{R}_+, \mathbb{R}^d)$.
Moreover, if Assumption \ref{ssfassu} also holds, then
\[\mathbb{P}^{{\rm CSRW}_n}_{\rho_n}\left(\left(X^{n,\omega,\nu}_{c_0 t/ a_nb_n^{1/\alpha}}\right)_{t\geq 0}\in\cdot\right)
\rightarrow\mathbb{P}^{{\rm FIN}}_{\rho}\left(\left(X^{\nu}_{t}\right)_{t\geq 0}\in\cdot\right)\]
weakly as probability measures on $D(\mathbb{R}_+,
\mathbb{R}^d)$.}

{\example\label{exa6-18} {\rm To continue with the example of the Sierpi\'nski gasket graphs from previous sections, one can also apply Proposition \ref{ssfrcmresult} for this collection. In particular, assuming that the conductances are uniformly bounded below and have at most polynomial decay at infinity, we know that nested fractals satisfy both Assumption \ref{thm:ass1} and \ref{thm:ass2}, and so we obtain the annealed convergence of the VSRW on the Sierpi\'nski gasket graphs as at (\ref{sgvsrw}). Moreover, if it is further the case that the tail behaviour at infinity of the conductances satisfies Assumption \ref{ssfassu}, then we also have the annealed convergence of the CSRW as at (\ref{sgcsrw}).}}

{\rem \label{finitemoments} {\rm When $\mathbf{E}\omega_e^0<\infty$ for each $e\in E_0$, one obtains in place of Lemma \ref{nunconvssf} (via the same argument) that there exists a constant $c_0$ such that $c_0^{-1}b_n\nu_n\rightarrow\mu$. Consequently, if Assumption \ref{ssfassu} is replaced by the assumption of finite first moments, then one can check the annealed limit of the CSRW is Brownian motion, rather than the FIN diffusion that appears in the second statement of Proposition \ref{ssfrcmresult}. A similar remark pertains to Proposition \ref{btmresult} and the second statement of Proposition \ref{rcmtreeresult}.}}

\providecommand{\bysame}{\leavevmode\hbox to3em{\hrulefill}\thinspace}
\providecommand{\MR}{\relax\ifhmode\unskip\space\fi MR }
\providecommand{\MRhref}[2]{%
  \href{http://www.ams.org/mathscinet-getitem?mr=#1}{#2}
}
\providecommand{\href}[2]{#2}


\begin{thebibliography}{10}

\bibitem{ADH}
R.~Abraham, J.-F. Delmas, and P.~Hoscheit, \emph{A note on the
  {G}romov-{H}ausdorff-{P}rokhorov distance between (locally) compact metric
  measure spaces}, Electron. J. Probab. \textbf{18} (2013), no. 14, 21.

\bibitem{AndHar}
S.~Andres and L.~Hartung, \emph{Diffusion processes on branching {B}rownian
  motion}, in preparation.

\bibitem{KajAnd}
S.~Andres and N.~Kajino, \emph{Continuity and estimates of the {L}iouville heat
  kernel with applications to spectral dimensions}, Probab. Theory Related
  Fields, to appear.

\bibitem{ALWtop}
S.~Athreya, W.~L\"{o}hr, and A.~Winter, \emph{The gap between {G}romov-vague
  and {G}romov-{H}ausdorff-vague topology}, preprint.

\bibitem{Barlow}
M.~T. Barlow, \emph{Diffusions on fractals}, Lectures on probability theory and
  statistics ({S}aint-{F}lour, 1995), Lecture Notes in Math., vol. 1690,
  Springer, Berlin, 1998, pp.~1--121.

\bibitem{BarCern}
M.~T. Barlow and J.~{\v{C}}ern{\'y}, \emph{Convergence to fractional kinetics
  for random walks associated with unbounded conductances}, Probab. Theory
  Related Fields \textbf{149} (2011), no.~3-4, 639--673.

\bibitem{BCK}
M.~T. Barlow, D.~A. Croydon, and T.~Kumagai, \emph{Subsequential scaling limits
  of simple random walk on the two-dimensional uniform spanning tree}, Ann.
  Probab., to appear.

\bibitem{BP}
M.~T. Barlow and E.~A. Perkins, \emph{Brownian motion on the {S}ierpi\'nski
  gasket}, Probab. Theory Related Fields \textbf{79} (1988), no.~4, 543--623.

\bibitem{BCCR}
G.~Ben~Arous, M.~Cabezas, J.~{\v{C}}ern{\'y}, and R.~Royfman, \emph{Randomly
  trapped random walks}, Ann. Probab. \textbf{43} (2015), no.~5, 2405--2457.

\bibitem{BC}
G.~Ben~Arous and J.~{\v{C}}ern{\'y}, \emph{Bouchaud's model exhibits two
  different aging regimes in dimension one}, Ann. Appl. Probab. \textbf{15}
  (2005), no.~2, 1161--1192.

\bibitem{Beres}
N.~Berestycki, \emph{Diffusion in planar {L}iouville quantum gravity}, Ann.
  Inst. Henri Poincar\'e Probab. Stat. \textbf{51} (2015), no.~3, 947--964.

\bibitem{BG}
R.~M. Blumenthal and R.~K. Getoor, \emph{Markov processes and potential
  theory}, Pure and Applied Mathematics, Vol. 29, Academic Press, New
  York-London, 1968.

\bibitem{CernyEJP}
J.~{\v{C}}ern{\'y}, \emph{On two-dimensional random walk among heavy-tailed
  conductances}, Electron. J. Probab. \textbf{16} (2011), no. 10, 293--313.

\bibitem{FukuChen}
Z.-Q. Chen and M.~Fukushima, \emph{Symmetric {M}arkov processes, time change,
  and boundary theory}, London Mathematical Society Monographs Series, vol.~35,
  Princeton University Press, Princeton, NJ, 2012.

\bibitem{Croyrogt}
D.~A. Croydon, \emph{Scaling limits for simple random walks on random ordered
  graph trees}, Adv. in Appl. Probab. \textbf{42} (2010), no.~2, 528--558.

\bibitem{CroyLT}
\bysame, \emph{Moduli of continuity of local times of random walks on graphs in
  terms of the resistance metric}, Trans. London Math. Soc. \textbf{2} (2015),
  no.~1, 57--79.

\bibitem{CroHamKum}
D.~A. Croydon, B.~M. Hambly, and T.~Kumagai, \emph{Heat kernel estimates for
  {FIN} diffusions associated with resistance forms}, in preparation.

\bibitem{CroyMuir}
D.~A. Croydon and S.~Muirhead, \emph{Functional limit theorems for the
  {B}ouchaud trap model with slowly varying traps}, Stochastic Process. Appl.
  \textbf{125} (2015), no.~5, 1980--2009.

\bibitem{DS}
B.~Duplantier and S.~Sheffield, \emph{Liouville quantum gravity and {KPZ}},
  Invent. Math. \textbf{185} (2011), no.~2, 333--393.

\bibitem{FIN}
L.~R.~G. Fontes, M.~Isopi, and C.~M. Newman, \emph{Random walks with strongly
  inhomogeneous rates and singular diffusions: convergence, localization and
  aging in one dimension}, Ann. Probab. \textbf{30} (2002), no.~2, 579--604.

\bibitem{FOT}
M.~Fukushima, Y.~Oshima, and M.~Takeda, \emph{Dirichlet forms and symmetric
  {M}arkov processes}, extended ed., de Gruyter Studies in Mathematics,
  vol.~19, Walter de Gruyter \& Co., Berlin, 2011.

\bibitem{GRV}
C.~Garban, R.~Rhodes, and V.~Vargas, \emph{Liouville {B}rownian motion}, Ann.
  Probab., to appear.

\bibitem{Garsia}
A.~M. Garsia, \emph{Continuity properties of {G}aussian processes with
  multidimensional time parameter}, Proceedings of the {S}ixth {B}erkeley
  {S}ymposium on {M}athematical {S}tatistics and {P}robability ({U}niv.
  {C}alifornia, {B}erkeley, {C}alif., 1970/1971), {V}ol. {II}: {P}robability
  theory, Univ. California Press, Berkeley, Calif., 1972, pp.~369--374.

\bibitem{GRR}
A.~M. Garsia, E.~Rodemich, and H.~Rumsey, Jr., \emph{A real variable lemma and
  the continuity of paths of some {G}aussian processes}, Indiana Univ. Math. J.
  \textbf{20} (1970/1971), 565--578.

\bibitem{GK}
R.~K. Getoor and H.~Kesten, \emph{Continuity of local times for {M}arkov
  processes}, Compositio Math. \textbf{24} (1972), 277--303.

\bibitem{HK}
B.~M. Hambly and T.~Kumagai, \emph{Heat kernel estimates for symmetric random
  walks on a class of fractal graphs and stability under rough isometries},
  Fractal geometry and applications: a jubilee of {B}eno\^\i t {M}andelbrot,
  {P}art 2, Proc. Sympos. Pure Math., vol.~72, Amer. Math. Soc., Providence,
  RI, 2004, pp.~233--259.

\bibitem{KL}
N.~Kajino, \emph{{N}eumann and {D}irichlet heat kernel estimates in inner
  uniform domains for local resistance forms}, in preparation.

\bibitem{Kall}
O.~Kallenberg, \emph{Foundations of modern probability}, second ed.,
  Probability and its Applications (New York), Springer-Verlag, New York, 2002.

\bibitem{Kigdendrite}
J.~Kigami, \emph{Harmonic calculus on limits of networks and its application to
  dendrites}, J. Funct. Anal. \textbf{128} (1995), no.~1, 48--86.

\bibitem{kig1}
\bysame, \emph{Analysis on fractals}, Cambridge Tracts in Mathematics, vol.
  143, Cambridge University Press, Cambridge, 2001.

\bibitem{Kig}
\bysame, \emph{Resistance forms, quasisymmetric maps and heat kernel
  estimates}, Mem. Amer. Math. Soc. \textbf{216} (2012), no.~1015, vi+132.

\bibitem{Kingman}
J.~F.~C. Kingman, \emph{Poisson processes}, Oxford Studies in Probability,
  vol.~3, The Clarendon Press, Oxford University Press, New York, 1993, Oxford
  Science Publications.

\bibitem{Kum}
T.~Kumagai, \emph{Heat kernel estimates and parabolic {H}arnack inequalities on
  graphs and resistance forms}, Publ. Res. Inst. Math. Sci. \textbf{40} (2004),
  no.~3, 793--818.

\bibitem{Kum2}
\bysame, \emph{Homogenization on finitely ramified fractals}, Stochastic
  analysis and related topics in {K}yoto, Adv. Stud. Pure Math., vol.~41, Math.
  Soc. Japan, Tokyo, 2004, pp.~189--207.

\bibitem{kk}
T.~Kumagai and S.~Kusuoka, \emph{Homogenization on nested fractals}, Probab.
  Theory Related Fields \textbf{104} (1996), no.~3, 375--398.

\bibitem{KZ}
T.~Kumagai and O.~Zeitouni, \emph{Fluctuations of maxima of discrete {G}aussian
  free fields on a class of recurrent graphs}, Electron. Commun. Probab.
  \textbf{18} (2013), no. 75, 12.

\bibitem{kus}
S.~Kusuoka, \emph{Statistical mechanics and fractals}, ch.~Lecture on diffusion
  processes on nested fractals, pp.~39--98, Springer Berlin Heidelberg, Berlin,
  Heidelberg, 1993.

\bibitem{lind}
T.~Lindstr{\o}m, \emph{Brownian motion on nested fractals}, Mem. Amer. Math.
  Soc. \textbf{83} (1990), no.~420, iv+128.

\bibitem{MRVZ}
P.~Maillard, R.~Rhodes, V.~Vargas, and O.~Zeitouni, \emph{Liouville heat
  kernel: regularity and bounds}, Ann. Inst. Henri Poincar\'e Probab. Stat., to
  appear.

\bibitem{MRa}
M.~B. Marcus and J.~Rosen, \emph{Sample path properties of the local times of
  strongly symmetric {M}arkov processes via {G}aussian processes}, Ann. Probab.
  \textbf{20} (1992), no.~4, 1603--1684.

\bibitem{MR}
\bysame, \emph{Markov processes, {G}aussian processes, and local times},
  Cambridge Studies in Advanced Mathematics, vol. 100, Cambridge University
  Press, Cambridge, 2006.

\bibitem{peir}
R.~Peirone, \emph{Convergence and uniqueness problems for {D}irichlet forms on
  fractals}, Boll. Unione Mat. Ital. Sez. B Artic. Ric. Mat. (8) \textbf{3}
  (2000), no.~2, 431--460.

\bibitem{sab}
C.~Sabot, \emph{Existence and uniqueness of diffusions on finitely ramified
  self-similar fractals}, Ann. Sci. \'Ecole Norm. Sup. (4) \textbf{30} (1997),
  no.~5, 605--673.

\bibitem{Whittpaper}
W.~Whitt, \emph{Some useful functions for functional limit theorems}, Math.
  Oper. Res. \textbf{5} (1980), no.~1, 67--85.

\end{thebibliography}
\end{document}